\documentclass[oneside, 11pt]{amsart}
\pdfoutput=1
\usepackage{amsmath, mathtools, amssymb, amsthm, amscd, alltt, graphicx, comment, bm}
\usepackage[utf8]{inputenc}
\usepackage[T1]{fontenc}
\usepackage[breaklinks=true,unicode]{hyperref}
\usepackage[capitalise]{cleveref}
\usepackage{tikz,tikz-cd}
\usepackage[a4paper, left=25mm, right=15mm, top=25mm, bottom=35mm]{geometry}
\usepackage{enumitem}
\usepackage{mathrsfs}

\numberwithin{equation}{section}
\newtheorem{lemma}{Lemma} \numberwithin{lemma}{section}
\newtheorem{proposition}[lemma]{Proposition}
\newtheorem{theorem}[lemma]{Theorem}
\newtheorem{corollary}[lemma]{Corollary}

\renewcommand{\Im}{\mathop{\mathrm{Im}}\nolimits}
\newcommand{\Ker}{\mathop{\mathrm{Ker}}\nolimits}
\newcommand{\K}{{\mathrm{K}}}
\newcommand{\St}{\mathop{\mathrm{St}}\nolimits}
\newcommand{\E}{\mathrm{E}}
\newcommand{\Gsc}{\mathrm{G}_\mathrm{sc}}

\newcommand{\Max}{\mathrm{Max}}

\theoremstyle{definition}
\newtheorem{definition}[lemma]{Definition} \Crefname{definition}{Definition}{Definitions}
\newtheorem{example}[lemma]{Example}

\theoremstyle{definition}
\theoremstyle{remark}
\newtheorem{rem}[lemma]{Remark}
\newtheorem*{conv}{Convention}

\newcommand{\ZZ}{\mathbb{Z}}

\newcommand{\rA}{\mathsf{A}}

\newcommand{\rC}{\mathsf{C}}
\newcommand{\rD}{\mathsf{D}}

\DeclareMathOperator{\colim}{colim}

\newcommand{\catname}[1]{{\normalfont\textbf{#1}}} 

\title{On the $\mathbb{A}^1$-invariance of $\mathrm K_2$ modeled on linear and even orthogonal groups}

\author{Andrei Lavrenov}
\address{St. Petersburg State University, 14th Line V.O., 29B, Saint Petersburg 199178 Russia}
\address{Bar-Ilan University, 5290002 Ramat Gan, Israel}
\email{avlavrenov at gmail.com}
\author{Sergey Sinchuk}
\address{St. Petersburg State University, 14th Line V.O., 29B, Saint Petersburg 199178 Russia}
\email{sinchukss at gmail.com}
\author{Egor Voronetsky}
\address{St. Petersburg State University, 14th Line V.O., 29B, Saint Petersburg 199178 Russia}
\email{voronetckiiegor at yandex.ru}

\date {\today}

\begin{document}

\maketitle

\begin{abstract}Let $k$ be an arbitrary field. In this paper we show that in the linear case ($\Phi=\rA_\ell$, $\ell \geq 4$) and even orthogonal case ($\Phi = \rD_\ell$, $\ell\geq 7$, $\mathrm{char}(k)\neq 2$) the unstable functor $\mathrm{K}_2(\Phi, -)$ possesses the $\mathbb{A}^1$-invariance property in the geometric case, i.\,e. $\K_2(\Phi, R[t]) = \K_2(\Phi, R)$ for a regular ring $R$ containing $k$. As a consequence, the unstable $\K_2$ groups can be represented in the unstable $\mathbb{A}^1$-homotopy category $\mathscr{H}_\bullet(k)$ as fundamental groups of the simply-connected Chevalley--Demazure group schemes $\mathrm{G}(\Phi,-)$. Our invariance result can be considered as the $\K_2$-analogue of the geometric case of Bass--Quillen conjecture. We also show for a semilocal regular $k$-algebra $A$ that $\K_2(\Phi, A)$ embeds as a subgroup into $\K^\mathrm{M}_2(\mathrm{Frac}\,A)$.
\end{abstract}

\section{Introduction}

Many approaches to higher algebraic $\K$-theory and Hermitian $\K$-theory are known, so one might be interested in finding comparison results for them.
For example, recall from~\cite[Theorem~IV.11.8]{Kbook} that stable Quillen's groups $\K_n(R)$ (defined e.\,g. via the $+$-construction) and stable Karoubi--Villamayor groups $\mathrm{KV}_n(R)$ coincide for $n\geq 1$ if $R$ happens to be a regular ring.
These theories use infinite-dimensional algebraic groups such as $\mathrm{GL}_\infty(R)$ in their definition. 
The aim of this work is to obtain an {\it unstable} analogue of such result for the functor $\K_2$.

Recall that one can define and study the analogues of Quillen's and Karoubi--Villamayor's $\K$-groups for finite-dimensional algebraic groups, see e.\,g.~\cite{Abe83, Jar83, Pa89, Re75, Sta14, Sta20, St71, St78, Su82, Tu83, VW16}. The interest for the unstable $\K_2$-groups, in particular, comes from the fact that they appear in Steinberg's presentation of the groups of points of algebraic groups by means of generators and relations. Let us recall this presentation in greater detail. Let $R$ be a commutative ring with a unit and $\Phi$ be an irreducible root system of rank $\geq 2$. The Steinberg group $\St(\Phi, R)$ is the group defined by means of generators $x_\alpha(\xi)$ (where $\alpha\in\Phi$, $\xi\in R$) and Chevalley commutator formulas, see~\cite[Ch.~6]{St67} (see also~\cref{sec:main-results}).
If one additionally chooses a lattice $\Lambda$ located between the root lattice $Q(\Phi)$ and the weight lattice $P(\Phi)$ of $\Phi$ (i.\,e. $Q(\Phi) \leq \Lambda \leq P(\Phi)$), one can construct the Chevalley--Demazure group scheme $\mathrm{G}_\Lambda(\Phi, -)$, see~\cite[\S~3]{St67}. The {\it unstable $\K_1$- and $\K_2$-functors modeled on the Chevalley group $ \mathrm{G}_\Lambda(\Phi, -)$} are defined via the following exact sequence (cf.~\cite{St78}):
\begin{equation} \label{eq:main-exactSeq} \begin{tikzcd} \K_2^{\mathrm{G}_\Lambda(\Phi, -)}(R) \ar[hookrightarrow]{r} & \St(\Phi, R) \ar{r}{\pi} & \mathrm{G}_\Lambda(\Phi, R) \ar[twoheadrightarrow]{r} & \K_1^{\mathrm{G}_\Lambda(\Phi, -)}(R). \end{tikzcd} \end{equation}
Thus, the group $\K_2^{\mathrm{G}_\Lambda(\Phi, -)}(R)$ can be interpreted as the group of ``nontrivial'' relations between the images of elements $x_\alpha(\xi)$ in $\mathrm{G}_\Lambda(\Phi, R)$, i.\,e. the relations depending on the arithmetic of $R$.
\begin{conv}
In the sequel we denote by $\Gsc(\Phi,-)$ the simply-connected form of the Chevalley--Demazure group scheme of type $\Phi$, i.\,e. $\mathrm{G}_\mathrm{sc}(\Phi, R) := \mathrm{G}_{P(\Phi)}(\Phi, R)$ and denote by $\K_i(\Phi, R)$ the corresponding $\K_1$ and $\K_2$ functors, i.e. $\K_i(\Phi, R) := \K_i^{\mathrm{G}_\mathrm{sc}(\Phi, -)}(R)$, $i=1,2$. This notation agrees with~\cite{St78}.
For shortness we also set $G := \mathrm{G}_\Lambda(\Phi, -)$.
\end{conv}
Our first main result is the following

\begin{theorem}[The $\K_2$-analogue of Lindel--Popescu theorem] \label{theorem:LP-for-K2}
Let $k$ be an arbitrary field and $R$ be a regular ring containing $k$. Let $\Phi$ be either $\rA_\ell$ for $\ell\geq4$, or $\rD_\ell$ for $\ell\geq 7$. In the latter case assume additionally that the characteristic of $k$ is not $2$. Then for any lattice $\Lambda$ as above and $G = \mathrm{G}_\Lambda(\Phi, -)$ one has $\K_2^G(R[t])\cong\K_2^G(R).$
\end{theorem} 

As a corollary of~\cref{theorem:LP-for-K2}, we show that the unstable groups $\K_2^G(R)$ coincide with Karoubi--Villamayor groups $\mathrm{KV}_2^G(R)$ (see~\cref{sec:K2-as-pi-1} for the definition). It follows from~\cite{AHW18} that the latter groups can be interpreted as fundamental groups of the pointed scheme $G$ in the unstable $\mathbb{A}^1$-homotopy category $\mathscr{H}_\bullet(k)$ of F.~Morel and V.~Voevodsky~\cite{MV99}. Thus, we obtain the following

\begin{corollary} \label{cor:motivic-pi1} (see Corollary~\ref{cor:motivic-pi1-2}) Let $\Phi$, $\Lambda$ and $k$ be as in~\cref{theorem:LP-for-K2}. Then for any $k$-smooth algebra $R$ and $G = \mathrm{G}_{\Lambda}(\Phi, -)$ one has
\[ \pi_1^{\mathbb{A}^1}(G)(R) := \mathrm{Hom}_{\mathscr{H}_{\bullet}(k)}(S^1 \wedge \mathrm{Spec}(R)_+, G) = \mathrm{KV}_2^{G}(R) = \K_2^G(R),\]
where $\mathscr{H}_\bullet(k)$ denotes the pointed unstable $\mathbb{A}^1$-homotopy category over $k$.
\end{corollary}

Our method of the proof of~\cref{theorem:LP-for-K2} also allows us to obtain the following assertion for the functor $\K_2(\Phi, R)$ as a byproduct, cf. e.\,g. with~\cite[Theorem~1.2]{Sta20}.
\begin{theorem} \label{theorem:Gersten} Let $\Phi$ and $k$ be as in~\cref{theorem:LP-for-K2} and let $A$ be an arbitrary semi-local regular domain containing $k$. Then the homomorphism $\K_2(\Phi, A) \to \K_2(\Phi, \mathrm{Frac}(A))$ induced by the embedding of $A$ into its field of fractions $\mathrm{Frac}(A)$, is injective. Moreover, the canonical homomorphism $\St(\Phi, A) \to \St(\Phi, \mathrm{Frac}(A))$ is injective and the group $\K_2(\Phi, A)$ embeds into the Milnor $\K_2$-group $\K_2^\mathrm{M}(\mathrm{Frac}(A))$. \end{theorem}

As yet another consequence of~\cref{theorem:LP-for-K2} we obtain the following results.
\begin{corollary} \label{cor:various-facts}
For $k$ with $\mathrm{char}(k)\neq 2$, $\ell \geq 7$ and $n\geq 0$ the group $\mathrm{Spin}_{2\ell}(k[t_1,\ldots, t_n])$ admits presentation as in~\cite[\S~6]{St67} or \cite[\S~5]{Ma69}, i.e. the presentation by means of generators $x_\alpha(\xi)$, $\xi \in k[t_1,\ldots, t_n]$ subject to relations~\eqref{R1}--\eqref{R3} and relations $h_\alpha(u)h_\alpha(v) = h_\alpha(uv)$, where $u, v\in k^\times$. Additionally, $\mathrm H_2(\mathrm{Spin}_{2\ell}(k[t_1,\ldots,t_n]),\,\mathbb Z\big) = \K^\mathrm{M}_2(k).$
\end{corollary}
\begin{corollary} \label{cor:H_2-O}
 For an arbitrary field $k$ of characteristic $\neq 2$, a regular $k$-algebra $R$, $\ell \geq 7$ one has
 \begin{align}
  \mathrm H_2 (\mathrm{SO}_{2\ell}(R[t]), \ZZ) =&\ \mathrm H_2 (\mathrm{SO}_{2\ell}(R), \ZZ); \label{eq:H_2-SO} \\
  \mathrm H_2 (\mathrm{O}_{2\ell}(R[t]), \ZZ) =&\ \mathrm H_2 (\mathrm{O}_{2\ell}(R), \ZZ). \label{eq:H_2-O}
 \end{align}
\end{corollary}

Recall that for $i \geq 0$ the $\mathbb{A}^1$-invariance for the stable linear functors $\K_i(R)$ and stable orthogonal functors $\mathrm{KO}_i(R)$ (see~\cite{Ho05} for the definition) is well-known for regular $R$, see e.\,g.~\cite[Theorem~V.6.3]{Kbook} and \cite[Corollary~1.12]{Ho05}, respectively. The classical Bass--Quillen conjecture asks if the $\mathbb{A}^1$-invariance for the functor $\K_0$ holds {\it unstably}, i.e. whether the set $\mathrm{VB}_n(R)$ of isomorphism classes of vector bundles of constant rank $n$ over $\mathrm{Spec}(R)$ has the property $\mathrm{VB}_n(R) = \mathrm{VB}_n(R[t])$. 
The classical Lindel--Popescu theorem for projective modules provides an affirmative answer to this question in the case when $R$ is regular ring containing a field $k$, see~\cite[\S~VIII.6]{Lam10},\cite{Li81, Po85}. In fact, the affirmative answer to the Bass--Quillen conjecture is known when $k$ is a Dedekind domain whose residue fields are perfect and $k \to R$ is a regular homomorphism in the sense of~\cite[\S~1]{Sw98}, see~\cite[Theorem~5.2.1]{AHW17}.
Thus, \cref{theorem:LP-for-K2} can be considered as the solution to the $\K_2$-analogue of the Bass--Quillen conjecture in the geometric case for the split linear and orthogonal groups.
Recall also that the $\mathbb{A}^1$-invariance for the unstable $\K_1$-functor modeled on Chevalley groups of rank $\geq 2$ has been established in~\cite{Abe83} and has been subsequently generalized to reductive groups of isotropic rank $\geq 2$ in the recent works of A.~Stavrova, see~\cite[Theorem~1.3]{Sta14}, \cite[Theorem~1.1]{Sta20}.

When $R=k$ is a field the assertion of~\cref{theorem:LP-for-K2} is known from~\cite{Re75} for all $\Phi$ (see Korollar of~Satz~1). Also if $\Phi = \rA_\ell$ and $\ell \geq \mathrm{max}(4, \mathrm{dim}(R) + 2)$ the assertion of~\cref{theorem:LP-for-K2} follows from the main result of~\cite{Tu83}. It is likely to remain true for all $\Phi$ of rank $\geq 3$ without any assumptions on the characteristic of $k$. It is false for $\Phi$ of rank $2$, however, see~\cite{We12}. The bottleneck in our proof which leads to the restrictive assumptions $\ell\geq 7$ and $\mathrm{char}(k)\neq 2$ in the case $\Phi=\rD_\ell$ is the orthogonal Horrocks theorem for $\K_2$ proved in \cite{LS20}, see Theorem~1. In our future papers we plan to address this issue and obtain an improved Horrocks theorem for all simply-laced $\Phi$ containing $\rA_4$ without the invertibility of $2$ assumption.

Recall from the~\cite[Theorem~5.3]{St71} that under the assumptions on the rank of $\Phi$ stated in~\cref{theorem:LP-for-K2} the first two homology groups of $\St(\Phi, R)$ vanish. On the other hand, the main result of~\cite{LSV20} asserts that $\K_2^G(R)$ is a central subgroup of $\St(\Phi, R)$ and, in particular, an abelian group.
From this and \S~IV.1 of~\cite{Kbook} it follows that the group $\K_2^G(R)$ coincides with the unstable $\K_2$-group defined by means of Quillen's $+$-construction:
\begin{equation} \label{eq:H2-K2}
  \K_{2}^{G,Q}(R) := \pi_2\left(\mathrm BG(R)^+_{\E_\Lambda(\Phi, R)}\right) = \mathrm{H}_2(\mathrm{E}_\Lambda(\Phi, R), \ZZ) = \K_2^G(R),
\end{equation}
where $\E_\Lambda(\Phi, R) = \mathrm{Im}(\pi)$ denotes the {\it elementary subgroup of $G(R) = \mathrm{G}_\Lambda(\Phi, R)$}  (see~\cref{sec:main-results}) and $\mathrm BG(R)$ is the topological classifying space of the discrete group $G(R)$. Thus, in particular, the definitions of the unstable $\K_2$-functors a la Quillen and a la Karoubi--Villamayor agree for any regular ring $R$ containing a field $k$ and $\Phi$ satisfying the assumptions of~\cref{theorem:LP-for-K2} (see Corolary~\ref{cor:+=A1}). For the more detailed exposition of Quillen's $+$-construction applied to Chevalley groups we refer the reader to~\cite[Subsection~2.3]{LS20}, more details on~(\ref{eq:H2-K2}) can also be found in the introduction of~\cite{LSV20}).

One of the ingredients in the proof of~\cref{theorem:LP-for-K2} is~\cref{lpb}, which gives a sufficient condition for a general group-valued functor to be $\mathbb{A}^1$-invariant. In turn, the proof of~\cref{lpb} is based on a recent result of I.~Panin, see~\cite[Theorem~2.5]{Pa19}. 

The sufficient condition of~\cref{lpb} is given as a short list of axioms. The hardest of these to verify is the so-called {\it $\mathbb{P}^1$-glueing property} of the unstable functor $\K_2$ (this property is sometimes also called Horrocks theorem, see e.\,g.~\S~IV.2, \S~VI.5 of \cite{Lam10}). In the unstable linear and orthogonal case the Horrocks theorem for $\K_2$ has been obtained in \cite{Tu83} and \cite{LS20}, respectively. Yet another axiom appearing in the statement of~\cref{lpb} is the so-called {\it weak affine Nisnevich excision property} for Steinberg groups. The verification of this property for simply-laced Steinberg groups is another central result of the paper, see~\cref{glueing}. Unlike M.~Tulenbaev, who used the technique of van der Kallen's ``another presentation'' in the proof of the weaker Zariski excision property~\cite[Proposition~1.4]{Tu83}, in our proof we use the new technique of pro-groups introduced by the third-named author in~\cite{Vor1} (see also~\cite{LSV20}). 

The assertion of Corollary~\ref{cor:motivic-pi1} is analogous to Stavrova's computation of $\pi_0$ of an isotropic reductive group, see~\cite[Theorem~5.5]{Sta20}. The analogue of Corollary~\ref{cor:motivic-pi1} in the broader context of reductive groups of isotropic rank $\geq 2$ (albeit only when $R = k$ is an infinite field) has also been obtained in~\cite{VW16}.
The representability of the stable $\K$-functors in $\mathscr{H}_\bullet(k)$ is well-known, see e.\,g. \cite[Corollary~3.4]{Ho05}, \cite[\S~4.3]{MV99}.

Notice also that the Nisnevich sheaf $\bm{\pi}_1^{\mathbb{A}^1}(G)$ on the category of $k$-smooth quasi-projective schemes $\catname{Sm}_k$ associated to the presheaf of groups $U \mapsto \pi_1^{\mathbb{A}^1}(G)(U)$, $U\in\catname{Sm}_k$ has been recently computed by F.~Morel and A.~Sawant for all split reductive $G$ without any assumptions on the rank of $G$. In the simply-connected case and under the assumptions on $\Phi$ from~\cref{theorem:LP-for-K2} this sheaf coincides with the unramified sheaf $\mathbf{K}_2^\mathrm{M}$ of Milnor $\K_2$-groups see~\cite[Theorem~1]{MS20}. 

In addition to~\cref{theorem:LP-for-K2} the proof of Corollary~\ref{cor:motivic-pi1} relies on the computation of the fundamental group of the simplicial group $\St(\Phi, R[\Delta^\bullet])$, see Proposition~\ref{prop:pi1-StDelta}. This computation can be considered as a more precise version of the calculation appearing in the proof of~\cite[Proposition~3.2]{VW16} in the setting of Chevalley groups.

\section{General formalism}
Throughout this section we denote by $\catname{Grp}$ the category of groups. For a commutative ring $k$ we denote by $\catname{Alg}_k$ the category of Noetherian commutative unital $k$-algebras. In the statement of~\cref{lpb}, which is the main result of this section, $k$ will be assumed to be a field. However, some of the intermediate steps in the proof of~\cref{lpb} can be obtained for $k=\ZZ$, so for the sake of greater generality we make no blanket assumption that $k$ is a field.  

For a commutative ring $R$ and an element $a\in R$ we denote by $\lambda_a$ the homomorphism of principal localization $R \to R_a$. Similarly, we denote by $\lambda_P$ the homomorphism $R \to R_P = (R\setminus P)^{-1}R$ of localization in a prime ideal $P \trianglelefteq R$. Notice that if $a$ is nilpotent then $R_a$ is the zero ring (which we also consider as an object of $\catname{Alg}_k$). 

Let $A$ be an $R$-algebra and $a \in A$. We denote by $ev_t(a)$ (or just $ev(a)$ when $t$ is clear from the context) the unique $R$-algebra homomorphism $R[t] \to A$ mapping $t$ to $a$. For a functor $K \colon \catname{Alg}_k \to \catname{Grp}$ and an element $g \in K(R[t])$ we often shorten the notation for the element $K(ev(a))(g)$ to just $g(a)$.

\begin{definition}\label{df:NK}
Let $K$ be a functor $\catname{Alg}_k \to \catname{Grp}$.
Denote by $\mathrm NK(R)$ the kernel of the homomorphism $K(R[t]) \to K(R)$ induced by the homomorphism of evaluation at $t=0$. It is clear that $K(R[t]) \cong \mathrm NK(R) \rtimes K(R)$ and $\mathrm NK$ is also a functor $\catname{Alg}_k \to \catname{Grp}$. If the group $K(R[t])$ happens to be abelian then $K(R)$ is also a normal subgroup of $K(R[t])$ and $K(R[t]) \cong \mathrm NK(R) \oplus K(R)$. Thus, our definition agrees with~\cite[Def.~III.3.3]{Kbook} if $K$ takes values in the subcategory $\catname{Ab}$ of abelian groups.
\end{definition}

For a commutative ring $R$ and $a \in R$ consider the following commutative square:
  \begin{equation} \label{M-sq} \begin{tikzcd} R \ltimes t R_a[t] \ar{r}{l} \ar{d}[swap]{e} & R_a[t] \ar{d}{ev(0)} \\ R \ar{r}{\lambda_a} & R_a. \end{tikzcd}\end{equation}
  The ring $R \ltimes tR_a[t]$ can be defined either via the semidirect product construction (see e.\,g. \cite[Definition~3.2]{S15}) or as the pullback of the diagram  $R \xrightarrow{\lambda_a} R_a \xleftarrow{ev(0)} R_a[t]$.
  It is also clear that~\eqref{M-sq} is a Milnor square, see~\cite[Example~I.2.6]{Kbook}.

Now suppose that $k$ is a field. Recall that a $k$-algebra $R$ is called {\it geometrically regular} over $k$, if for any finite field extension $E/k$ the ring $R\otimes_kE$ is regular (cf.~\cite[p.~137]{Sw98}). In particular, a geometrically regular algebra is Noetherian. If $k$ is perfect then a $k$-algebra $R$ is geometrically regular over $k$ if and only if it is regular.

The following result gives a sufficient condition for a functor $K$ to be $\mathbb{A}^1$-invariant, cf. e.\,g. with~\cite[Proposition~2.2]{AHW20}.
\begin{theorem} \label{lpb}
 Let $k$ be a field.
 Suppose that a functor $K \colon \catname{Alg}_k \to \catname{Grp}$ satisfies the following axioms:
 \begin{enumerate}[label=\textnormal{(A\arabic*)}]
  \item \label{CFC} {\it $K$ is finitary}, i.e. commutes with filtered colimits. In other words if $A_i$ are Noetherian $k$-algebras and their colimit $A$ is also a Noetherian algebra, then $K(A) = \colim_i(K(A_i))$.
  \item \label{DP} For a $k$-algebra $R$ and $a \in R$ consider the diagram obtained from~\eqref{M-sq} by applying $K$. Then the homomorphism $\Ker\left(K(e)\right) \to \mathrm NK(R_a)$ between the kernels of vertical arrows is injective.
  \item \label{LPP} {\it $K$ satisfies weak affine Nisnevich excision for domains.} By this we mean the following. Let $\iota \colon B \hookrightarrow A$ be an injective homomorphism of domains contained in $\catname{Alg}_k$ such that $\mathrm{Spec}(A)\to\mathrm{Spec}(B)$ is an {\'e}tale morphism of affine schemes. Let $h$ be an element of $B$ not invertible in $A$ such that $\iota$ induces an isomorphism $B / hB \cong A / \iota(h)A$. Consider the following commutative square: \[\begin{tikzcd} B \ar{r}{\iota} \ar{d}{\lambda_h} & A \ar{d}{\lambda_{\iota(h)}}\\ B_h \ar{r}{\overline{\iota}} & A_h \end{tikzcd}\] Then the natural homomorphism \[\Ker(K(B) \to K(B_h)) \to \Ker(K(A) \to K(A_h))\] induced by $\iota$ is surjective.
  \item \label{PGP} {\it $K$ satisfies $\mathbb{P}^1$-glueing property for local domains.} By definition, this means that for every local domain $R \in \catname{Alg}_k$ the following diagram whose arrows are induced by natural embeddings is a pullback square: \begin{equation}\label{eq:P1-square} \begin{tikzcd} K(R) \ar[r] \ar[d] & K(R[t]) \arrow{d} \\ K(R[t^{-1}]) \ar{r} & K(R[t, t^{-1}]). \end{tikzcd} \end{equation}  
  It is easy to see that if~\eqref{eq:P1-square} is pullback, then all its arrows are injective.
  \item \label{HIF} {\it $K$ is homotopy invariant for fields}, i.e. $\mathrm NK$ vanishes on every field $F \in \catname{Alg}_k$.
 \end{enumerate}
 Then $\mathrm NK$ vanishes on every geometrically regular $R\in \catname{Alg}_k$. In other words, for every such $R$ the natural embedding $R \hookrightarrow R[t]$ induces an isomorphism $K(R)\cong K(R[t]).$
\end{theorem}
\begin{rem}
 In the literature a functor satisfying the axiom~\ref{PGP} is sometimes called {\it acyclic} (cf.~\cite[Definition~III.4.1.1]{Kbook}).
The choice of the name for~\ref{LPP} is inspired by axiom (P3) from~\cite[Proposition~3.3.4]{AHW18}. Notice that if $B$ Noetherian of finite Krull dimension then the application of $\mathrm{Spec}$ to the commutative square from~\ref{LPP} produces a distinguished Nisnevich square in $\catname{Sm}/\mathrm{Spec}(B)$ in the sense of~\cite[Definition~3.1.3]{MV99}. Notice also that our axioms can be slightly relaxed, so that the assertion of the Theorem becomes slightly stronger, see Remark~\ref{rem:relax} below.
\end{rem}

The proof is based on a series of lemmas and is deferred until the end of this section. We start by deducing the following useful corollary of the axiom~\ref{DP}. 

Let $k$ be an arbitrary commutative ring and $K\colon \catname{Alg}_k \to \catname{Grp}$ be a functor.
We say that $K$ satisfies the {\it Quillen--Suslin local-global principle} if for every $R\in \catname{Alg}_k$ the following map is injective:
\begin{equation} \label{QS-def} \begin{tikzcd} \mathrm NK(R) \ar{r}{\prod \lambda_M} & \prod\limits_{M \in \Max(R)} \mathrm NK(R_M). \end{tikzcd} \end{equation}
\begin{lemma}\label{LGP}
Suppose that $K$ is a finitary functor satisfying~\ref{DP}. Then $K$ satisfies the Quillen--Suslin local-global principle. In particular, $\mathrm NK$ is a separated presheaf in the Zariski topology.
\end{lemma}

\begin{proof}
 Fix a $k$-algebra $R$ and $a \in R$. 
 Consider the following diagram of $k$-algebras. Its objects $A_i$ are copies of $R[t]$ indexed by natural numbers $i$. The only arrows of this diagram are the homomorphisms $ev (a^{j-i}t)\colon A_i \to A_j$ defined for $1 \leq i \leq j$. It is clear that this diagram is filtered and its colimit is $R \ltimes tR_a[t]$ (cf. e.\,g.~\cite[Lemma~15]{S15}). 
 
 The first step of the proof of the lemma is to show that for every $g \in \Ker(\mathrm NK(R) \to \mathrm NK(R_a))$ there exists some natural $n$ such that $g(a^nt)$ is the trivial element of $K(R[t])$. Let $g$ be such an element. From~\ref{DP} we obtain that the image of $g$ in $K(R \ltimes tR_a[t])$ is trivial. The required assertion now follows from~\ref{CFC} and the previous paragraph.
 
 The next step of the proof is to verify that for any coprime elements $a, b \in R$ the homomorphism $\langle K(\lambda_a), K(\lambda_b) \rangle \colon \mathrm NK(R) \to \mathrm NK(R_a) \times \mathrm NK(R_b)$ is injective. Fix $g \in \Ker(\langle K(\lambda_a), K(\lambda_b) \rangle)$. We argue as in the proof of~\cite[Lemma~2.5]{Tu83}. Set $S := R[t, t_1]$. Consider the element $h(t, t_1, t_2) = g(t_1 t) \cdot g((t_1 + t_2)t)^{-1}\in K(S[t_2]).$ Clearly, $h$ lies in the kernel of the homomorphism $K(ev_{t_2}(0))$. Since evaluation commutes with localization, we obtain that the element $K(\lambda_a)(h) \in K(S_a[t_2])$ is trivial, therefore by the previous paragraph there exists $n$ such that $h(t, t_1, a^nt_2)$ is trivial in $K(S[t_2])$. Similarly, we find $m$ such that $h(t, t_1, b^m t_2)$ is trivial. Since $a^n$ and $b^m$ are still coprime, we can find $x, y \in R$ such that $xa^n + yb^m = 1$. The required assertion now follows from the following calculation:
 $$1 = h(t, 1, -yb^m) \cdot h(t, xa^n, -xa^n) = g(t)\cdot g(xa^n\cdot t)^{-1} \cdot g(xa^n\cdot t) \cdot g(0)^{-1} = g(t).$$
 
 Now we can finish the proof of the lemma. We argue as in the proof of~\cite[Theorem~2]{S15}. Let $g$ be an element of the kernel of~\eqref{QS-def}. Denote by $Q(g)$ the set consisting of all elements $c \in R$ for which $K(\lambda_c)(g)$ is trivial. It is clear that $Q(g)$ is closed with respect to multiplication by elements of $R$. Let us check that this set is, in fact, an ideal. Fix $a, b \in Q(g)$ and let $c$ be an element of the ideal $\langle a, b \rangle$. Notice that $\overline{a} = \lambda_c(a)$, $\overline{b} = \lambda_c(b)$ are coprime elements of $R_c$. From the identities $\lambda_{\overline{a}}\lambda_c = \lambda_{\lambda_a(c)}\lambda_a$ and $\lambda_{\overline{b}}\lambda_c = \lambda_{\lambda_b(c)}\lambda_b$ we obtain that $g' = K(\lambda_c)(g)$ lies in the kernel of $\langle K(\lambda_{\overline{a}}), K(\lambda_{\overline{b}}) \rangle.$ By the previous paragraph, we obtain that $g' = 1$ and hence that $c \in Q(g)$. We have shown that $Q(g)$ is an ideal of $R$. If $Q(g)$ is proper then it is contained in a maximal ideal $M \trianglelefteq R$, in which case from $K(\lambda_M)(g) = 1$ and~\ref{CFC} we obtain that $K(\lambda_s)(g) = 1$ for some $s \in R \setminus M$. Thus, we obtain a contradiction, so $Q(g) = R$ and $g = 1$.
\end{proof}

First of all, notice that the local-global principle allows one to obtain the following global version of~\ref{PGP} (cf. e.\,g. with the proof of~\cite[Theorem~1]{LS20}). 
\begin{lemma}[$\mathbb{P}^1$-glueing property for domains] \label{ght}
Suppose that $K$ satisfies the Quillen--Suslin local-global principle and the $\mathbb{P}^1$-glueing property for local domains. Then for an arbitrary domain $R \in \catname{Alg}_k$ the square~\eqref{eq:P1-square} is a pullback square. In particular, all its arrows are injective. \end{lemma}
\begin{proof}
Let $R \in \catname{Alg}_k$ be an arbitrary domain. Suppose that the images of $g \in K(R[t])$ and $h \in K(R[t^{-1}])$ coincide in $K(R[t, t^{-1}])$. Then so do the images of $g' = g \cdot g(0)^{-1}$ and $h' = h \cdot g(0)^{-1}$. For a maximal ideal $M\trianglelefteq R$ set $g'_M := K(\lambda_M)(g') \in K(R_M[t])$ and $h'_M := K(\lambda_M)(h') \in K(R_M[t^{-1}]))$. Clearly, their images in $K(R_M[t, t^{-1}])$ coincide. By our assumptions, the square~\eqref{eq:P1-square} is pullback for $R_M$, therefore $g'_M = 1$ for all $M$. Since $g' \in \mathrm NK(R)$, by the local-global principle we obtain that $g' = 1$, hence $g$ is the image of some element of $K(R)$. By symmetry, $h$ also is the image of some element of $K(R)$. Since $K(R)\to K(R[t, t^{-1}])$ is injective, we conclude that $g$ and $h$ are the images of the same element of $K(R)$, which completes the proof.  \end{proof}

Let us also note the following useful particular special case of weak affine Nisnevich excision.
\begin{lemma}[weak affine Zariski excision for domains]
	\label{zgl} Suppose that a functor $K\colon \catname{Alg}_k \to \catname{Grp}$ satisfies~\ref{LPP}. Let $R \in \catname{Alg}_k$ be a domain and $a, b$ be a pair of coprime elements of $R$. Consider the diagram
$$\begin{tikzcd}
	R \ar{r}{\lambda_a} \ar{d}[swap]{\lambda_b} & R_a \ar{d}{\overline{\lambda_b}}\\
	R_b \ar{r}{\overline{\lambda_a}} & R_{ab}.
\end{tikzcd}$$
	Then the natural map $\Ker(K(\lambda_b)) \to \Ker(K(\overline{\lambda_b}))$ induced by $\lambda_a$ is surjective.
\end{lemma}
\begin{proof}
	Set $B=R$, $A=R_a$ and $h=b$. It is clear that $\mathrm{Spec}(A)\to\mathrm{Spec}(B)$ is an {\'e}tale morphism. To see that the present situation is a special case of~\ref{LPP} it is enough to  check that the map $j\colon B/hB \to A/hA$ is an isomorphism. Let us first verify the surjectivity of $j$, or what is the same, the inclusion $A \subseteq Ah+B$. Fix an element $r/a^s\in A=R_a$, we need to show that it lies in $Ah+B$. We may assume that $s\geq 1$, so that $a^s$ and $b^s$ are still coprime. Choose $x, y \in R$ such that $xa^s+yb^s=1$. Thus, $r/a^s$ can be decomposed into the sum of $rx\in B$ and $ry(b/a)^s\in bR_a=hA$.

	Now let us verify the injectivity of $j$, i.e. the inclusion $Ah\cap B \subseteq Bh$. Suppose that $rb/a^s=c\in B=R$. We may assume $s\geq 1$ otherwise there is nothing to prove. From $rb=a^sc$ and $xa^s+yb^s=1$ we obtain the required inclusion $c=cxa^s+cyb^s=xrb+cyb^s=(xr+cyb^{s-1})b \in Rb = Bh$. 
\end{proof}

The following injectivity result is inspired by~\cite[Corollary~5.2]{Tu83}.
\begin{lemma} \label{lmp}
Suppose that $K$ satisfies weak affine Zariski excision and $\mathbb{P}^1$-glueing properties for domains. Then for any domain $R\in \catname{Alg}_k$ and any monic polynomial $f\in R[t]$ the localization homomorphism $\lambda_f\colon R[t]\rightarrow R[t]_f$ induces an injection $K(R[t])\hookrightarrow K(R[t]_f).$ \end{lemma}
\begin{proof}
	Fix a presentation $f=\sum_{i=0}^n a_it^i$, in which $a_n=1$. Set $$g=1+a_{n-1}t^{-1}+\ldots+a_0t^{-n}\in R[t^{-1}].$$ It is clear that $R[t, t^{-1}]_f \cong R[t, t^{-1}]_g$. Consider the following commutative diagram:
$$\begin{tikzcd}
	R[t] \ar{r}{\lambda_t} \ar{d}{\lambda_f} & R[t, t^{-1}] \ar{d}{\overline{\lambda}} & R[t^{-1}] \ar{l}[']{\lambda_{t^{-1}}} \ar{d}{\lambda_g}\\
	R[t]_f \ar{r} & R[t, t^{-1}]_f & R[t^{-1}]_g. \ar{l}
\end{tikzcd}$$
Let $x$ be an element of $\Ker(K(\lambda_f))$. Then $K(\lambda_t)(x)$ lies in $\Ker(K(\overline{\lambda}))$ so by Lemma~\ref{zgl} applied to the right square we find $y \in \Ker(K(\lambda_{g}))$ such that $K(\lambda_t)(x) = K(\lambda_{t^{-1}})(y)$. Now by Lemma~\ref{ght} $x$ and $y$ are images of some $z \in K(R)$. 
Since the composite homomorphism $K(R) \to K(R[t^{-1}]_g) \xrightarrow{K(ev_{t^{-1}}(0))} K(R)$ coincides with the identity map, we conclude that $z=1$, consequently $x = y = 1$ and the proof is complete. \end{proof}

For the rest of this section we assume that $k$ is a field.
Recall that a $k$-algebra $R$ is called {\it essentially smooth} if it is geometrically regular and essentially of finite type over $k$. 
Equivalently, $R$ is essentially smooth over $k$ if it is essentially of finite type over $k$ and $R\otimes_k\overline k$ is regular (see e.\,g.~\cite[p.~137]{Sw98}). 
If, moreover, $R$ is of finite type over $k$, it is called {\it smooth over $k$}.

\begin{theorem}[Panin]\label{thm:Panin}
\label{paninthm}
	Let $k$ be a field, $R$ be a domain smooth over $k$. Let $M_1,\ldots, M_n$ be a finite set of maximal ideals of $R$. Denote by $A=R_{M_1,\ldots,M_n}$ the corresponding semi-localization of $R$. Let $f$ be an element of the intersection  $\cap_{i=1}^nM_i \subseteq R$. Then there exist a monic polynomial $h(t)\in A[t]$, a domain $S$ essentially smooth over $k$ and homomorphisms $\tau$, $p$, $p'$ and $\delta$ such that $\tau^*\colon\mathrm{Spec}(S)\to\mathrm{Spec}(A[t])$ is an {\'e}tale morphism and the following diagram commutes:
\begin{equation}\label{eq:panin-diag}
 \begin{tikzcd}[column sep=4em]
   & A & \\ A[t] \ar{d}{\lambda_h} \ar{ru}{ev(0)} \ar{r}{\tau} & S \ar{u}{\delta} \ar{d}{\lambda_{\tau(h)}}  & R \ar{ul}[swap]{\lambda_{M_1,\ldots,M_n}} \ar{d}{\lambda_f} \ar{l}[swap]{p} \\
   A[t]_h \ar{r}[swap]{\tau\otimes_{A[t]}A[t]_h}              & S_{\tau(h)}      & R_f. \ar{l}{p'}\end{tikzcd}
\end{equation}
Additionally, the homomorphism $\tau$ is an injection and one has $A[t]/hA[t]\cong S/\tau(h)S.$ 
\end{theorem}

The geometric content of this theorem can be heuristically explained as follows. Let $X=\mathrm{Spec}(R)$ be $k$-smooth affine scheme and $X_f=\mathrm{Spec}(R_f)$ its principal open set, as before $A= R_{M_1,\ldots, M_n}$. The left-hand square of the diagram~(\ref{eq:panin-diag}) defines an isomorphism of motivic spaces (i.e. Nisnevich sheaves) $\mathrm{Spec}(S)/\mathrm{Spec}(S_{\tau(h)})\cong(\mathbb A^1\times\mathrm{Spec}(A))/(\mathbb A^1\times\mathrm{Spec}(A))_h$. Then the theorem asserts that the canonical map $\mathrm{Spec}\,A\to X/X_f$ factors through $(\mathbb A^1\times\mathrm{Spec}\,A)/(\mathbb A^1\times\mathrm{Spec}\,A)_h$. We refer the reader to~\cite[Section~2]{Pa19} for more details.

\begin{proof}
 All of the assertions except the last two are a direct ring-theoretic restatement of (i)--(iii) of~\cite[Theorem~2.5]{Pa19}.
	Since $\tau^*$ is open, $S\neq 0$ and $A[t]$ is a domain, the image of $\tau^*$ is dense. Denote by $I$ the kernel of $\tau$. The closed subscheme defined by $I$ must contain the image of $\tau^*$ therefore $I$ is radical and hence zero.

	Set $X = \mathrm{Spec}(A[t]/h)$, $Y = \mathrm{Spec}(S/h)$. We need to show that the morphism $\tau' \colon Y \to X$ obtained from $\tau^*$ by restriction is an isomorphism. Set $Y_0 := Y \times_{X} X_\mathrm{red}$. From the fact that the left square is an elementary distinguished Nisnevich square we obtain that the composite morphism $Y_\mathrm{red} \not\hookrightarrow Y_0 \xrightarrow{\tau''} X_\mathrm{red}$ is an isomorphism, consequently since $\tau''$ is {\'e}tale, the closed embedding $(Y_0)_\mathrm{red} = Y_\mathrm{red} \not\hookrightarrow Y_0$ is also {\'e}tale, hence an open embedding, hence an isomorphism (see e.\,g. Corollary~3.6, Proposition~3.10 of~\cite{Mi80}). Thus, we conclude that $\tau''$ is an isomorphism and hence by the topological invariance of {\'e}tale morphisms~\cite[Theorem~3.23]{Mi80}, so is $\tau'$.
\end{proof}

\begin{corollary}
\label{esssmooth}
Let $k$ be a field.
Suppose that $K\colon\catname{Alg}_k\rightarrow\catname{Grp}$ is finitary and satisfies weak affine Nisnevich excision and $\mathbb{P}^1$-glueing properties. Let $R, M_1, \ldots, M_n, A$ be as in the statement of the above theorem. Denote by $E$ the fraction field of $A$ {\rm(}which coincides with the fraction field of $R${\rm)}.
Then for any $m\geq 0$ the natural homomorphism $K(A[x_1,\ldots, x_m])\to K(E[x_1,\ldots,x_m])$ is injective.
\end{corollary}
\begin{proof}
We argue as in the proof of~\cite[Theorem~3.2]{Sta20}. Set $B=k[x_1,\ldots x_m]$. Let $g$ be an element of the kernel of $K(B \otimes_k A)\rightarrow K(B \otimes_k E)$. First of all, notice that both $A$ and $E$ are filtered colimits of principal localizations of $R$. Since $K$ is finitary, there exists $f' \in \cap_{i=1}^n(R \setminus M_i)$ such that $g$ is the image of some $g_1 \in K(R_{f'}[x_1,\ldots x_m])$ under $K(\lambda_{M_1',\ldots M_n'})$, where $M_i' = R_{f'} \cdot M_i$. Set $R' := R_{f'}$.
Clearly, there exists $f \in R'$ such that $K(\lambda_{f})(g_1) = 1$ in $K(R'_{f})$. 
Without loss of generality, we may assume that $f \in \cap_{i=1}^n M_i'$. 

We apply Theorem~\ref{paninthm} to the ring $R'$, maximal ideals $M_i'$ and the polynomial $f$ as above.
Tensoring~\eqref{eq:panin-diag} with $B$ and applying functor $K$ we obtain the following commutative diagram (we use the convention that tensoring with $B$ does not change the notation for the arrows of the diagram):
\begin{equation*} 
 \begin{tikzcd}[column sep=4em]
   & K(B \otimes_k A) & \\ 
   K(B \otimes_k A[t]) \ar{d}{K(\lambda_h)} \ar{ru}{K(ev_t(0))} \ar{r}{K(\tau)} & K(B \otimes_k S) \ar{u}{K(\delta)} \ar{d}{K(\lambda_{\tau(h)})}  & K(B \otimes_k R') \ar{ul}[swap]{K(\lambda)} \ar{d}{K(\lambda_f)} \ar{l}[swap]{K(p)} \\
   K(B \otimes_k A[t]_h) \ar{r}[swap]{}              & K(B \otimes_k S_{\tau(h)})      & K(B \otimes_k R'_f). \ar{l}{}\end{tikzcd}
\end{equation*}
Since $K(\lambda_f)(g_1) = 1$ the element $K(p)(g_1)$ lies in the kernel of $K(\lambda_{\tau(h)})$. Thus, by~\ref{LPP} there exists $g_2\in K(B \otimes_k A[t])$ such that $K(\lambda_h)(g_2)=1$ and $K(\tau)(g_2)=K(p)(g_1)$. Since $h \in A[t]$ is monic, by Lemma~\ref{lmp} the homomorphism $K(\lambda_h)$ is injective, therefore $g_2=1$. It remains to see that
$$g=K(\lambda)(g_1)=K(\delta)(K(p)(g_1))=K(\delta)(K(\tau)(g_2))=g_2(0)=1.\qedhere$$
\end{proof}

\begin{rem}
For the proof of~\cref{lpb} we only need the special cases $m=0,1$ of the above result.
\end{rem}

\begin{theorem}[Popescu]
\label{popescu} Let $k$ be a field, and $R$ a ring geometrically regular over $k$. Then $R$ is a filtered colimit of smooth $k$-algebras. \end{theorem}
\begin{proof} See~\cite{Po85}, \cite[Theorem~1.1]{Sw98}.
\end{proof}

Now we are ready to finish the proof of \cref{lpb}. 
\begin{proof}[Proof of~\cref{lpb}]
Our goal is to prove the triviality of $\mathrm NK(R)$.
By~\cref{popescu} $R$ is a filtered colimit of smooth $k$-algebras.
Since filtered colimits commute with finite limits and the functor $K$ is finitary, the functor $\mathrm NK$ is also finitary.
Thus, it suffices to verify the triviality of $\mathrm NK(R)$ for a smooth $k$-algebra $R$.
Further, by Lemma~\ref{LGP} we are left to prove that $\mathrm NK(R_M)$ is trivial for every maximal ideal $M$ in a such an algebra.
Since every smooth scheme over $k$ is a disjoint union of smooth irreducible $k$-schemes, we may assume, without loss of generality, that $R$ is a domain.
Denote by $E$ the fraction field of $R_M$. Consider the diagram
\[\begin{tikzcd}
K\bigl(R_M[t]\bigr) \ar[hookrightarrow]{r} \ar{d}{ev_t(0)} & K(E[t]) \ar{d}{\cong}\\
K\bigl(R_M\bigr) \ar[hookrightarrow]{r} & K(E),
\end{tikzcd}\]
in which the right vertical arrow is an isomorphism by~\ref{HIF} and the horizontal arrows are injective by Corollary~\ref{esssmooth}.
Thus, the left vertical arrow is also injective and $\mathrm NK(R_M) = 1$, as required.
\end{proof}
\begin{rem}\label{rem:relax}
It is clear from the above proof that it possible to weaken the requirements on $K$ in the statement of~\cref{lpb}. For example, we could require that axioms \ref{DP}--\ref{PGP} hold only for essentially smooth $k$-algebras. Further, the assertion of the axiom~\ref{CFC} can be restricted to algebras geometrically regular over $k$.
\end{rem}

We can also prove the following result (cf.~\cite[Th{\'e}or{\`e}me~1.1]{CTO92}, cf. also with~\cite[Proposition~2.2]{AHW20}).
\begin{theorem}\label{thm:gb}
Let $k$ be a field.
Suppose that $K\colon\catname{Alg}_k\rightarrow\catname{Grp}$
is a finitary functor satisfying weak affine Nisnevich excision and $\mathbb{P}^1$-glueing properties for domains. 
Let $A$ be a semi-local domain geometrically regular over $k$. 
Denote by $E$ the fraction field of $A$.
Then the natural homomorphism $K(A) \to K(E)$ is injective.
\end{theorem}
\begin{proof}
By Theorem~\ref{popescu} $A$ is a filtered colimit of smooth $k$-algebras. Since $A$ is a domain we may assume additionally that these smooth $k$-algebras are domains.

Since $K$ is finitary, we may assume that $g \in \Ker(K(A) \to K(E))$ comes from some element $g_1\in\mathrm K(R)$, where $R$ is a smooth $k$-domain. We may assume additionally that $g_1$ lies in the kernel of the natural homomorphism $K(R) \to K(\mathrm{Frac}(R))$.

Denote by $P_i$ the preimages of the maximal ideals of $A$ and choose some maximal ideals $M_i$ of $R$ such that $M_i$ contains $P_i$. Then by Corollary~\ref{esssmooth}, $g_1$ vanishes in the semi-localization of $R$ in the set of maximal ideals $M_i$, therefore it vanishes in the semi-localization of $R$ in the set of prime ideals $P_i$. Since $R\rightarrow A$ factors through the latter semi-localization, we conclude that $g=1$. 
\end{proof}

\section{Proof of the main results} \label{sec:main-results}
\subsection{Proof of Theorems~\ref{theorem:LP-for-K2} and~\ref{theorem:Gersten}}
Throughout this section $\Phi$ is an irreducible root system of rank $\geq 2$ and $R$ is a commutative ring with a unit. Recall that in the case when $\Phi$ is simply-laced the Steinberg group $\St(\Phi, R)$ can be defined by means of the generators $x_\alpha(\xi)$, $\xi \in R$ and the following set of defining relations:
\begin{align}
x_{\alpha}(a)\cdot x_{\alpha}(b)&=x_{\alpha}(a+b),\text{ for } \alpha\in \Phi;\tag{R1} \label{R1}\\
[x_{\alpha}(a),\,x_{\beta}(b)]  &=1,\text{ for }\alpha+\beta\not\in\Phi\cup 0; \tag{R2} \label{R2} \\
[x_{\alpha}(a),\,x_{\beta}(b)]  &=x_{\alpha+\beta}(N_{\alpha\beta} \cdot ab)\text{ in the case }\alpha+\beta\in\Phi. \tag{R3} \label{R3} \end{align}
The integers $N_{\alpha,\beta} = \pm 1$ appearing in the last identity are called the {\it structure constants} of the Chevalley group $\mathrm{G}_\mathrm{sc}(\Phi, R)$. They coincide with the structure constants of the corresponding simple Lie algebra. More information about Steinberg groups, structure constants and further references can be found e.\,g. in~\cite[\S~3]{St71}, \cite[\S~2.4]{LSV20}. Notice that the above definition of the Steinberg group makes sense even if $R$ does not have a unit, this will be important later in~\cref{sec:patching}.

We denote by $\E_\Lambda(\Phi, R)$ the {\it elementary subgroup} of $G = \mathrm{G}_\Lambda(\Phi, R)$. It is defined as the abstract subgroup generated by elementary root unipotents $t_\alpha(\xi)$, $\alpha \in \Phi$, $\xi \in R$, see e.\,g.~\cite{St71, St78, VZ20, Abe83}. It is clear that $\E_\Lambda(\Phi, R)$ is the image of the homomorphism $\pi$ from~\eqref{eq:main-exactSeq}. We write $\E_\mathrm{sc}(\Phi, R)$ instead of $\E_{P(\Phi)}(\Phi, R)$.

For $a, b \in R$ set $y_\alpha(a, b) = [x_\alpha(a), x_{-\alpha}(b)]$.
For an ideal $I \trianglelefteq R$ we denote by $\overline{\St}(\Phi, R, I)$ the normal closure in $\St(\Phi, R)$ of the subgroup generated by $x_\alpha(a)$, $a\in I$. Recall from~\cite[Lemma~5]{S15} that
\begin{equation} \label{eq:relative-st-ker}
 \overline{\St}(\Phi, R, I) = \Ker(\St(\Phi, R) \to \St(\Phi, R/I)).
\end{equation}

\begin{lemma} \label{lem:c-identities} For an arbitrary irreducible root system $\Phi$ of rank $\geq 2$, arbitrary ideals $A, B \trianglelefteq R$ and all $a \in A$, $b \in B$, $c \in R$ the following congruences hold:
\begin{itemize}
 \item $y_\alpha(a, cb) \equiv y_\alpha(ac, b)\ (\mathrm{mod}\ \overline{\St}(\Phi, R, AB))$ in the case when either $\Phi \neq \rC_{\ell}$ or $\alpha$ is short.
 \item $y_\alpha(a, c^2b) \equiv y_\alpha(ac^2, b)\ (\mathrm{mod}\ \overline{\St}(\Phi, R, AB))$, $y_\alpha(a, cb)^2 \equiv y_\alpha(ac, b)^2\ (\mathrm{mod}\ \overline{\St}(\Phi, R, AB))$ in the case $\Phi = \rC_{\ell}$ and $\alpha$ is long.
\end{itemize} \end{lemma}
\begin{proof}
 Observe that the proof of~\cite[Theorem~5]{VZ20} is based solely on computations with Chevalley commutator formula, which all can be reproduced verbatim for Steinberg groups.
\end{proof}

Notice that the Steinberg group functor {\it does not} commute with general finite limits. However, it satisfies the following weaker property, which we are going to use in the sequel.
\begin{lemma} \label{lem:fprod} For an arbitrary irreducible root system $\Phi$ of rank $\geq 2$ the Steinberg group functor $\St(\Phi, -)$ commutes with finite direct products. \end{lemma}
\begin{proof} 
It suffices to verify the assertion for binary direct products.
Observe that canonical projections $R_1 \times R_2 \to R_i$, $i=1,2$ split in the category of rings without unit, therefore the groups $G_i = \St(\Phi, R_i)$ embed as subgroups into $G = \St(\Phi, R_1 \times R_2)$. It is also clear that $G_i$ together generate $G$. Thus, to verify the isomorphism $G \cong G_1 \times G_2$ it suffices to show the triviality of the commutator subgroup $[G_1, G_2] \leq G$.

Set $A = R_1\times 0$, $B = 0 \times R_2$. It is clear that $A, B \trianglelefteq R_1 \times R_2$. 
It follows directly from~\eqref{R2}--\eqref{R3} that the commutators $[x_{\alpha}(a),\ x_\beta(b)]$ are trivial for all $\beta \neq -\alpha$, $a \in A$, $b\in B$. On the other hand, to verify the triviality of $y_\alpha(a, b)$ we can apply the congruences of Lemma~\ref{lem:c-identities} (since $AB=0$ these congruences turn into equalities).
Indeed, setting $c = c^2 = (1, 0)$ we obtain that $y_\alpha(a, b) = y_\alpha(ac, b) = y_\alpha(a, bc) = y_\alpha(a, 0) = 1$. \end{proof}

\begin{lemma} \label{k2cdc} 
For an arbitrary root system $\Phi$ the functors $\Gsc(\Phi,\,-)$, $\St(\Phi,\,-)$ and $\K_2(\Phi,\,-)$ commute with filtered colimits.
\end{lemma}
\begin{proof}
The assertion for $\mathrm G_{\mathrm{sc}}(\Phi,\,-)$ follows from the fact that it is represented in the category $\catname{Ring}$ by a finitely presented (Hopf) $\ZZ$-algebra (by e.\,g.~\cite[\href{https://stacks.math.columbia.edu/tag/00QO}{Tag~00QO}]{stacks-project}). The assertion for $\St(\Phi, -)$ is obvious from its definition. The assertion for the functor $\K_2(\Phi, -)$ now follows from the assertions for $\St(\Phi, -)$ and $\Gsc(\Phi, R)$ using the fact that filtered colimits commute with finite limits in $\catname{Grp}$ (in particular, they commute with kernels).
\end{proof}

\begin{proof}[Proof of \cref{theorem:LP-for-K2}]
Let us verify that the axioms \ref{CFC}--\ref{HIF} are satisfied for the functor $\K_2(\Phi, -)$. Axiom~\ref{CFC} is satisfied by Lemma~\ref{k2cdc}. Axiom~\ref{DP} follows from the fact that for a simply-laced $\Phi$ of rank $\geq 3$ Steinberg groups $\St(\Phi, -)$ satisfy Tulenbaev lifting property in the sense of~\cite[Definition~2.1]{LS17} (more precisely, \ref{DP} follows from the injectivity of the left arrow in the diagram~(2.2) ibid.).
The axiom~\ref{LPP} is checked in~\cref{sec:patching} below.
In the case $\Phi = \rA_\ell$ the axiom~\ref{PGP} is satisfied by~\cite[Theorem~5.1]{Tu83}. In the orthogonal case it is satisfied by~\cite[Theorem~1]{LS20}. Finally, axiom~\ref{HIF} is satisfied by the Korollar of \cite[Satz~1]{Re75}.

Notice that every regular ring containing a field $k$ is geometrically regular over the prime subfield of $k$ (which is always perfect). Now the assertion of~\cref{theorem:LP-for-K2} in the simply-laced case ($\Lambda = P(\Phi)$) follows from~\cref{lpb}.

It remains to verify that the assertion of~\cref{theorem:LP-for-K2} holds for the Chevalley group $G = \mathrm{G}_\Lambda(\Phi, -)$ of non-simply-connected type. Notice that there is an exact sequence of elementary subgroups
\begin{equation} \label{eq:E-seq} \begin{tikzcd} 1 \ar[r] & M(R) \ar[r] & \E_\mathrm{sc}(\Phi, R) \ar[r] & \E_\Lambda(\Phi, R), \end{tikzcd} \end{equation}
where $M$ is either $\mu_n$ for some $n$ or $\mu_2\times \mu_2$ (cf.~\cite[\S~3]{St67}).
Applying the snake lemma to the diagram
$$
\begin{tikzcd} 
  &  1\ar[r]\ar[d]  & \mathrm{St}(\Phi, R)   \ar{d} \ar[r]  & \mathrm{St}(\Phi, R) \ar{d} \ar[r]  & 1 \\
      1  \ar[r] & M(R)\ar[r] & \E_\mathrm{sc}(\Phi, R) \ar[r] &  \E_\Lambda(\Phi, R) &
    \end{tikzcd}
$$
and the diagram obtained from the above one by replacing $R$ with $R[t]$ we obtain the rows of the diagram
\begin{equation}\label{eq:snake-A1} \begin{tikzcd} 
      1  \ar[r] & \K_2(\Phi, R)   \ar{d}{\cong} \ar[r]  & \K_2^G(R) \ar{d}{i_G} \ar[r] & M(R) \ar[r] \ar{d}{i_M} & 1 \\
      1  \ar[r] & \K_2(\Phi, R[t])\ar[r] & \K_2^G(R[t]) \ar[r] & M(R[t]) \ar[r] & 1.
    \end{tikzcd} \end{equation}
We already know that the left vertical arrow is an isomorphism. The homomorphism $i_M$ is also an isomorphism since $R$ is reduced. Thus, $i_G$ is an isomorphism by 5-lemma.
\end{proof}

\begin{proof}[Proof of \cref{theorem:Gersten}]
  Set $E := \mathrm{Frac}(A)$. The injectivity of $\K_2(\Phi, A) \to \K_2(\Phi, E)$ follows from~\cref{thm:gb}. Since $\mathrm{G}_\mathrm{sc}(\Phi, A) \to \mathrm{G}_\mathrm{sc}(\Phi, E)$ is injective, this implies that $\St(\Phi, A) \to \St(\Phi, E)$ is also injective. Finally, the last assertion follows from the fact that the group $\K_2(\Phi, E)$ is isomorphic to the Milnor group $\K_2^\mathrm{M}(E)$ by Matsumoto's theorem~\cite[Theorem~5.10]{Ma69}.
\end{proof}

\begin{rem}\label{rem:many-units}
We keep the notation of~\cref{theorem:Gersten}. The injectivity of $\K_2(\Phi, A) \to \K_2(\Phi, E)$ in the linear case or under the additional assumption that $A$ does not have finite residue fields can be deduced from previously known results. Indeed, Popescu's \cref{popescu} and Quillen's theorem~\cite[Theorem~V.9.6]{Kbook} imply the injectivity of $\K_2(A) \to \K_2(E)$ (cf.~\cite[Theorem~A]{Pa03}). By Matsumoto's theorem $\K_2(\Phi, E) \cong \K_2(E)$. We claim that under our assumptions $\K_2(\Phi, A) \cong \K_2(A)$. Indeed, in the linear case this follows from the main result of~\cite{ST81}. If $A$ does not have finite residue fields this follows from~\cite[Example~1.5d and Theorem~3.7]{vdK77}. On the other hand, if $A$ possesses finite residue fields and $\Phi=\rD_\ell$ the assertion of~\cref{theorem:Gersten} is new.\end{rem}

\begin{proof}[Proof of Corollary~\ref{cor:various-facts}]
  Set $R = k[t_1,\ldots t_n]$. Recall from~\cite{Abe83} that $\K_1(\rD_\ell, R) = \K_1(\rD_\ell, k) = 1$, so $\mathrm{Spin}_{2\ell}(R) = \E_\mathrm{sc}(\rD_\ell, R)$. The presentation a la Steinberg for $\mathrm{Spin}_{2\ell}(R)$ now follows from~\cref{theorem:LP-for-K2} and Matsumoto's theorem~\cite[Theorem~5.10]{Ma69}. The equality $\mathrm H_2(\E_\mathrm{sc}(\rD_\ell, R),\,\mathbb Z) = \K_2(\rD_\ell, R)$ is the main result of~\cite{LS17}. \end{proof}
 
 \begin{proof}[Proof of Corollary~\ref{cor:H_2-O}]
  We will give the proof for~\eqref{eq:H_2-O}, the proof for~\eqref{eq:H_2-SO} being similar.
  From the homological Lyndon--Hochschild--Serre spectral sequence associated with the short exact sequence $\mathrm{EO}_{2\ell}(R) \hookrightarrow \mathrm{O}_{2\ell}(R) \twoheadrightarrow \mathrm{KO}_{1, 2\ell}(R)$ we obtain an exact sequence
  \[ \begin{tikzcd} \mathrm H_3(\mathrm{KO}_{1,2\ell}(R)) \ar[r] & \mathrm H_2(\mathrm{EO}_{2\ell}(R))_{\mathrm{KO}_{1,2\ell}(R)} \ar[r] & \mathrm H_2(\mathrm{O}_{2\ell}(R)) \ar[r] &  \mathrm H_2(\mathrm{KO}_{1, 2\ell}(R)) \ar[r] & 1, \end{tikzcd} \]
  where the lower index in the second term denotes the coinvariants of the action of $\mathrm{KO}_{1,2\ell}(R)$.
  The functor $\mathrm{KO}_{1, 2\ell}(-)$ is $\mathbb{A}^1$-invariant (e.\,g. by~\cite[Theorem~1.1]{Sta20}).
  The functor $\mathrm H_2(\mathrm{EO}_{2\ell}(-))$ clearly coincides with $\mathrm{KO}_{2, 2\ell} := \K_2^{\mathrm {SO}_{2\ell}}$ and is also $\mathbb{A}^1$-invariant by~\cref{theorem:LP-for-K2}. Thus, by 5-lemma we conclude that the central term is also $\mathbb{A}^1$-invariant.
 \end{proof}

\subsection{\texorpdfstring{$\K_2(\Phi, R)$}{K2(R)} as $\mathbb{A}^1$-fundamental group} \label{sec:K2-as-pi-1}
Recall from~\cite{Jar83} that for an arbitrary commutative (unital) ring $R$ one can define the standard simplicial ring $R[\Delta^\bullet]$ as follows:
\begin{equation}
 R[\Delta^n] = R[t_0,\ldots t_n]\left/\left\langle \sum_{i=0}^n t_i -1 \right\rangle,\right.\ d_i(t_j) = \left\{ \begin{array}{ll}t_j & j < i, \\0 & j = i, \\ t_{j-1} & j > i; \end{array}\right. s_i(t_j) = \left\{ \begin{array}{ll} t_j & j < i, \\ t_j + t_{j+1} & j = i, \\ t_{j+1} & j > i. \end{array} \right.
\end{equation}
	The application of a functor $G\colon\catname{Ring}\rightarrow\catname{Grp}$ to $R[\Delta^\bullet]$ yields a simplicial group $G(R[\Delta^\bullet])$, which in the context of $\mathbb{A}^1$-homotopy theory is usually called the {\it simplicial resolution} of $G$ and is denoted by $\mathrm{Sing}^{\mathbb{A}^1}_\bullet(G)(R)$. Recall also, that the $n$-th {\it unstable Karoubi--Villamayor K-functor attached to $G$} (denoted $\mathrm{KV}^G_n(R)$) is, by definition, the $n-1$-th simplicial homotopy group of $G(R[\Delta^\bullet])$ (cf.~\cite[\S~4.3]{AHW18}, \cite[\S~3]{Jar83}).

Recall from~\cite[\S~17]{May67} that the homotopy group $\pi_n(G)$ of a simplicial group $(G_\bullet, d_i, s_i)$ can be computed as $n$-th homology group of the normalized Moore complex
\[ 1 \leftarrow N_0(G) \xleftarrow{\partial_1} N_1(G) \xleftarrow{\partial_2} N_2(G) \xleftarrow{\partial_3} \ldots\,, \]
in which $N_n(G) = \cap_{i=1}^n\Ker(d_i) \trianglelefteq G_n$ and the differential $\partial_k$ is obtained from $d_0\colon G_k \to G_{k-1}$ by restriction, and $\Im(\partial_{n+1})\trianglelefteq\Ker(\partial_n)$ (see~\cite[Propositions~17.3--17.4]{May67}). In other words, $\pi_n(G) \cong \mathrm{H}_n(N_\bullet G) = \Ker(\partial_n) / \Im(\partial_{n+1})$.

\begin{proposition}\label{prop:pi1-StDelta} Let $\Phi$ be an arbitrary irreducible root system. Then for any commutative ring $R$ the simplicial groups $\E_\Lambda(\Phi,\,R[\Delta^\bullet])$ and $\St(\Phi,\,R[\Delta^\bullet])$ are connected. If, moreover, $\Phi$ has rank at least $2$ then the simplicial group $\St(\Phi, R[\Delta^\bullet])$ is simply-connected. \end{proposition}
\begin{proof}
In our computations we identify the ring $R[\Delta^n]$ with $R[t_1, \ldots, t_n]$ via $t_0 = 1 - \sum_{i=1}^n t_i$ and use~\eqref{eq:relative-st-ker} repeatedly.
The map $d_1\colon\St(\Phi,\,R[\Delta^1])\rightarrow\St(\Phi,\,R[\Delta^0])$ is given by $t_1\mapsto0$, therefore \[N_1\St(\Phi,\,R[\Delta^\bullet])=\Ker(d_1)=\overline{\St}(\Phi,\,R[t_1],\,\langle t_1\rangle),\]
and $\partial_1\colon N_1\St(\Phi,\,R[\Delta^\bullet])\rightarrow N_0\St(\Phi,\,R[\Delta^\bullet])=\St(\Phi,\,R)$ is induced by $d_0$, which sends $t_1$ to $t_0=1$. 
It is clear that $x_\alpha(r)$ of $\St(\Phi, R)$ is the image of $x_\alpha(rt_1)$ under $\partial_1$.
This shows that $\pi_0\St(\Phi, R[\Delta^\bullet])=1$, i.e. $\St(\Phi, R[\Delta^\bullet])$ is connected. The argument for $\E_\Lambda(\Phi, R[\Delta^\bullet])$ is identical.

Now let us verify the second assertion. Notice that the kernel of $\partial_1$ coincides with the intersection $\overline{\St}(\Phi,\,R[t_1],\,\langle t_1\rangle)\cap\overline{\St}(\Phi,\,R[t_1],\,\langle t_1-1 \rangle )$, or, what is the same, with the kernel of the homomorphism
$\St(\Phi,\,R[t_1])\rightarrow\St(\Phi,\,R)\times\St(\Phi,\,R)$
sending $g(t_1)$ to $\big(g(0),\,g(1)\big)$. By Lemma~\ref{lem:fprod} this homomorphism can be identified with the homomorphism $\St(\Phi, R[t_1]) \to \St(\Phi, R\times R)$ induced by the ring homomorphism of evaluation of $t_1$ at $(0, 1)$.

Since the ideals $\langle t_1 \rangle$ and $\langle t_1-1 \rangle$ are coprime, we can identify $R\times R$ with $R[t_1]/t_1(t_1-1)$ by the Chinese remainder theorem. Thus, we obtain that $\Ker(\partial_1)$ coincides with $\overline{\St}(\Phi,\,R[t_1],\,\langle t_1(t_1-1) \rangle)$ and, in particular, is generated by $x_\alpha\big(t_1(t_1-1)f(t_1)\big)^{g(t_1)}$, where $\alpha \in \Phi$, $f\in R[t_1]$, $g \in \St(\Phi, R[t_1])$.

Notice that the face maps $d_1, d_2\colon\mathrm{St}(\Phi,\,R[\Delta^2])\rightarrow\mathrm{St}(\Phi,\,R[\Delta^1])$ are given by evaluations ($t_1\mapsto0$, $t_2\mapsto t_1$) and ($t_1\mapsto t_1$, $t_2\mapsto0$), respectively. Thus, we obtain that \[N_2\mathrm{St}(\Phi,\,R[\Delta^\bullet])=\Ker(d_1)\cap\Ker(d_2)=\overline{\St}(\Phi,\,R[t_1,\,t_2],\,\langle t_1\rangle )\cap\overline{\St}(\Phi,\,R[t_1,\,t_2],\,\langle t_2\rangle).\]
The differential $\partial_2$ is induced by the face map $d_0 \colon \St(\Phi, R[\Delta^2]) \to \St(\Phi, R[\Delta^1])$. The latter, in turn, is given by the evaluation ($t_1 \mapsto 1-t_1$, $t_2 \mapsto t_1$).

It remains to see that the elements $x_{\alpha}\big(t_1t_2\,f(t_2)\big)^{g(t_2)}$ belong to $N_2\St(\Phi,\,R[\Delta^\bullet])$ and are mapped by $\partial_2$ onto the generating set of $\Ker(\partial_1)$ mentioned above. Thus, the normalized Moore complex for $\St(\Phi, R[\Delta^\bullet])$ is exact at $N_1$-term, which completes the proof of the proposition. \end{proof}

Let $f\colon G_\bullet\twoheadrightarrow Q_\bullet$ be a degreewise surjective morphism of simplicial groups (i.e. $f_n\colon G_n\to Q_n$ are surjective for all $n$). Recall from~\cite[Theorem~1.3]{Ina75} that in this situation the degreewise kernel $K_\bullet$ (i.e. the simplicial group given by $K_n = \Ker(f_n)$ with face and degeneracy maps induced from those of $G_\bullet$) yields a long exact sequence of groups
\begin{equation} \label{eq:simplicial-les} \begin{tikzcd} \ldots \ar{r} & \pi_{1}(G_\bullet) \ar{r} & \pi_1(Q_\bullet) \ar{r} & \pi_0(K_\bullet) \ar{r} & \pi_0(G_\bullet). \end{tikzcd} \end{equation}

\begin{theorem} \label{theorem:pi1-GRDelta}
 For $\Phi$, $\Lambda$, $k$ and $R$ as in~\cref{theorem:LP-for-K2} and $G = \mathrm{G}_\Lambda(\Phi, -)$ one has $\pi_1(G(R[\Delta^\bullet])) = \K_2^G(R)$.
\end{theorem}
\begin{proof}
First of all, notice that by the homotopy invariance for $\K_1$ (see e.\,g.~\cite[Theorem~1.1]{Sta20}) the simplicial group $\K_1^G(R[\Delta^\bullet])$ is discrete. Applying the exact sequence~\eqref{eq:simplicial-les} to the canonical morphism $G(R[\Delta^\bullet]) \twoheadrightarrow \K_1^G(R[\Delta^\bullet])$, we obtain that $\pi_1(G( R[\Delta^\bullet])) \cong \pi_1(\E_\Lambda(\Phi, R[\Delta^\bullet]))$.

Now consider the simplicial map $\mathrm{st}_\bullet \colon \St(\Phi, R[\Delta^\bullet]) \twoheadrightarrow \E_\Lambda(\Phi, R[\Delta^\bullet])$ given by canonical projections $\mathrm{st}$ in each degree. The application of~\eqref{eq:simplicial-les} yields an exact sequence of groups
\[
\pi_1\bigl(\St(\Phi, R[\Delta^\bullet])\bigr) \to \pi_1\bigl(\E_\Lambda(\Phi, R[\Delta^\bullet])\bigr) \to \pi_0\bigl(\K_2^G( R[\Delta^\bullet])\bigr) \to \pi_0\bigl(\St(\Phi, R[\Delta^\bullet])\bigr).
\]
 Proposition~\ref{prop:pi1-StDelta} implies that the first and the last groups in this exact sequence are trivial, so the two central groups are isomorphic. On the other hand, \cref{theorem:LP-for-K2} implies that $\K_2^G(R[\Delta^\bullet])$ is discrete, so $\pi_0(\K_2^G(R[\Delta^\bullet])) = \K_2^G(R)$, as claimed.
\end{proof}

\begin{corollary}[Corollary~\ref{cor:motivic-pi1}] \label{cor:motivic-pi1-2} Let $\Phi$, $\Lambda$ and $k$ be as in~\cref{theorem:LP-for-K2}. Then for any smooth $k$-algebra $R$ and $G = \mathrm{G}_{\Lambda}(\Phi, -)$ one has
\[ \pi_1^{\mathbb{A}^1}(G)(R) := \mathrm{Hom}_{\mathscr{H}_{\bullet}(k)}(S^1 \wedge \mathrm{Spec}(R)_+, G) = \mathrm{KV}_2^{G}(R) = \K_2^G(R),\]
where $\mathscr{H}_\bullet(k)$ denotes the pointed unstable $\mathbb{A}^1$-homotopy category over $k$.
\end{corollary}
\begin{proof}
 By~\cite[Corollary~5.4]{Sta20} for every regular $k$-algebra $R$ one has \[\mathrm H^1_\mathrm{Nis}(R, G) = \mathrm H^1_\mathrm{Nis}(R[t], G),\] consequently by~\cite[Theorem~2.4.2]{AHW18} the scheme $G$ defines an $\mathbb{A}^1$-naive presheaf on $\mathrm{Sm}_{k}^{\mathrm{aff}}$ in the sense of~\cite[Definition~2.1.1]{AHW18}.
 In particular, one has $\pi_1^{\mathbb{A}^1}(G)(R) = \pi_1(G(R[\Delta^\bullet]))$, the latter group being isomorphic to $\K_2^G(R)$ by the above theorem.
\end{proof} 

\begin{corollary}
\label{cor:+=A1}
Let $\Phi$, $\Lambda$ and $k$ be as in~\cref{theorem:LP-for-K2}. Then for any regular ring $R$ containing a field $k$ and $G = \mathrm{G}_{\Lambda}(\Phi, -)$ one has
$$
\K_{2}^{G,Q}(R) = \mathrm{KV}_2^G(R),
$$
where $\K_2^{G, Q}(R)$ denotes the $\mathrm K_2$-group defined via Quillen's $+$-construction, see~(\ref{eq:H2-K2}) or formula~\cite[(2.21)]{LS20}. In other words, the definitions of the unstable $\K_2$-functors a la Quillen and a la Karoubi--Villamayor agree for regular rings containing a field.
\end{corollary}

\section{Nisnevich glueing for \texorpdfstring{$\K_2(\Phi, R)$}{K2(Ф,R)}} \label{sec:patching}
The main result of this section is the following glueing theorem for Steinberg groups.
\begin{theorem}\label{glueing}
Let $\iota\colon B \hookrightarrow A$ be an embedding of integral domains, $h\in B \setminus \{0\}$ be such that $B / h = A / h$. Denote by $\overline\iota\colon B_h\hookrightarrow A_h$ the induced embedding of localizations. Let $\Phi$ be an arbitrary simply-laced root system of rank at least $3$.
Then $\St(\Phi, B)$ surjects onto the pullback of the diagram 
$\St(\Phi, A) \xrightarrow{\lambda_h} \St(\Phi, A_{h})\xleftarrow{\overline{\iota}}\St(\Phi, B_h)$:

 \[ \begin{tikzcd}
 \St(\Phi, B)\ar[two heads, bend left=5]{rd}\ar[bend left=10]{rrd}{\lambda_h}\ar["\iota"', bend right=20]{rdd}&&\\
&\St(\Phi, B_h)\times_{\St(\Phi, A_{h})}\St(\Phi, A)\ar{r}\ar{d}&\St(\Phi, B_h)\ar{d}{\overline{\iota}}\\
&\St(\Phi, A)\ar{r}{\lambda_{h}}&\St(\Phi, A_{h}).
\end{tikzcd} \]
\end{theorem}

The above result generalizes~\cite[Proposition~1.4]{Tu83}, which is formulated for the Zariski topology and the special case $\Phi=\rA_{\geq 4}$.
The analogues of the above theorem for the functor $\K_1$ have also been obtained by E.~Abe and A.~Stavrova, cf.~\cite[Lemma~3.7]{Abe83}, \cite[Lemma~3.4]{Sta14}.
The analogous result for projective modules is the ``descent lemma'' \cite[Lemma~4.7]{Bh99} which is attributed to W.~L\"utkebohmert, H.~Lindel and M.~Kumar. 

\begin{corollary}
In the notation of Theorem~\ref{glueing} the natural map
$$
\mathrm{Ker}\big(\St(\Phi,\,B)\rightarrow\St(\Phi,\,B_h)\big)\rightarrow\mathrm{Ker}\big(\St(\Phi,\,A)\rightarrow\St(\Phi,\,A_h)\big)
$$
is surjective. In particular the functors $\St(\Phi, -)$, $\K_2(\Phi, -)$ satisfy weak affine Nisnevich excision for domains in the sense of Theorem~\ref{lpb}\,(\ref{LPP}).
\end{corollary}
\begin{proof}
Take $x\in\mathrm{Ker}\big(\St(\Phi,\,A)\rightarrow\St(\Phi,\,A_h)\big)$, then by Theorem~\ref{glueing} there exists a pre-image $y\in\St(\Phi,\,B)$ of the element $(1,\,x)\in\St(\Phi, B_h)\times_{\St(\Phi, A_{h})}\St(\Phi, A)$, i.e. $\lambda_h(y)=1$ and $\iota(y)=x$. In other words, $y$ is a pre-image of $x$ from $\mathrm{Ker}\big(\St(\Phi,\,B)\rightarrow\St(\Phi,\,B_h)\big)$. Since $\iota\colon\mathrm{G_{sc}}(\Phi,\,B)\rightarrow\mathrm{G_{sc}}(\Phi,\,A)$ is injective, we also get the assertion for $\mathrm K_2(\Phi,\,-)$.
\end{proof}

\subsection{Plan of the proof}
\label{plan}

The assertion of Theorem~\ref{glueing} is equivalent to the exactness of the sequence of pointed sets 
\begin{equation}\label{stavrova-sequence} \begin{tikzcd} \St(\Phi, B) \ar{rrr}{g \mapsto \big(\lambda_h(g)^{-1},\, \iota(g)\big)} & &  & \St(\Phi, B_h) \times \St(\Phi, A) \ar{rrrr}{(g_1,g_2) \mapsto \overline{\iota}(g_1)\cdot\lambda_{h}(g_2)} & & & & \St(\Phi, A_{h}) \end{tikzcd} \end{equation}
in the central term. 
\begin{definition}
Consider the natural action of $\St(\Phi, B)$ on $\St(\Phi, B_h) \times \St(\Phi, A)$ given by
\[g \star (u, v) = (u \lambda_h(g)^{-1}, \iota(g) v),\text{ where }g \in \St(\Phi, B),\ u \in \St(\Phi, B_h),\ v \in \St(\Phi, A),\]
and let $V$ denote the orbit set of this action. We use the notation $[u, v]\in V$ for the orbit corresponding to a pair $(u, v) \in \St(\Phi, B_h) \times \St(\Phi, A)$.
\end{definition}

Consider the natural map $V \to \St(\Phi, A_h)$ induced by $\overline{\iota}$ and $\lambda_h$. If we show that this map is injective the exactness of~\eqref{stavrova-sequence} would follow. The main ingredient in the proof of this injectivity is the following

\begin{proposition}
\label{action}
There exists a well-defined action $\cdot$ of $\St(\Phi, A_{h})$ on $V$ such that the map $V \to \St(\Phi, A_{h})$ is $\St(\Phi, A_{h})$-equivariant (here $\St(\Phi, A_{h})$ acts on itself by left translation). Moreover, for any $u\in\St(\Phi, B_h)$ and $v\in\St(\Phi, A)$ one has 
\begin{align}
\label{additional}
\lambda_{h}(v)\cdot[1,\,1]=[1,\,v]\quad\text{ and }\quad\overline\iota(u)\cdot[1,\,v]=[u,\,v].
\end{align}
\end{proposition}
\begin{corollary}
\label{corollary-in-the-end}
The natural map $V\to \St(\Phi, A_{h})$ is bijective, i.e. $V$ is a principal homogeneous set for $\St(\Phi, A_{h})$.
\end{corollary}
\begin{proof}
Since $[1,\,1]\in V$ is mapped to $1\in\St(\Phi,\,A_{h})$, we conclude that for any $g\in\St(\Phi,\,A_{h})$ the element $g\cdot[1,\,1]\in V$ is mapped to $g$, i.e. the above map is surjective. On the other hand, if $[u,\,v]$ and $[u',\,v']$ are mapped to the same element of $\St(\Phi,\,A_{h})$, i.e. $\overline\iota(u)\lambda_{h}(v)=\overline\iota(u')\lambda_{h}(v')$, then using~(\ref{additional}) we have
$$
[u,\,v]=\overline\iota(u)\cdot[1,\,v]=\overline\iota(u)\lambda_{h}(v)\cdot[1,\,1]=\overline\iota(u')\lambda_{h}(v')\cdot[1,\,1]=[u',\,v'],
$$
i.e. the above map is injective.
\end{proof}

In the remainder of this subsection we outline the construction of the action from Proposition~\ref{action}. Recall from Section~3 that the definition of the Steinberg group $\St(\Phi,\,R)$ makes sense for a non-unital $R$.

{\bf Step~1.} For $g\in\St(\Phi,\,A_h)$ and sufficiently large $k$ we define a homomorphism $c_{g}\colon\St(\Phi,\,h^kA)\rightarrow\St(\Phi,\,A)$ which models conjugation by $g$ on the level of Steinberg groups (in particular, for all $x\in \St(\Phi,\, h^kA)$ one has the identity $\pi(c_g(x)) = \pi(g) \pi(x) \pi(g^{-1})$ in $\mathrm{G}_{\mathrm{sc}}(\Phi, A_h)$). 
The map $c_g$ is in fact already constructed in~\cite{LSV20} using the technique of pro-groups which is discussed in Subsections~\ref{pro-completion}-\ref{pro-groups}. We also remark that this construction is not canonical.

\begin{rem}
The idea behind the construction of $c_g$ is quite simple. It is enough to define the maps $\St(\Phi,\,h^{k+t}A)\rightarrow\St(\Phi,\,h^tA)$ modeling the conjugation by $g=x_\beta(a/h^s)$ for sufficiently large $k$ (here $k$ depends on $s$ but not on $t$, $a\in A$). We set
$$
c_{x_\beta(a/h^s)}\big(x_\gamma(bh^{k+t})\big)=
\begin{cases}
x_{\gamma}(bh^{k+t})&\text{  for }\beta+\gamma\not\in\Phi\cup0,\\
x_{\beta+\gamma}(N_{\beta\gamma}abh^{k+t-s})x_{\gamma}(bh^{k+t})&\text{ for }\beta+\gamma\in\Phi.
\end{cases}
$$
This agrees with Steinberg relations \eqref{R2} and \eqref{R3}. The difficulty here is to define  $c_{x_\beta(a/h^s)}(x_{-\beta}(b/h^{k+t}))$, but this can be done by decomposing $x_{-\beta}(b/h^{k+t})$ into a commutator using the relation~\eqref{R3}. We refer the reader to the introduction of~\cite{LSV20} for a more detailed explanation of this construction.
\end{rem}

{\bf Step~2.} For $\alpha\in\Phi$, $c/h^s\in A_h$ we construct a map $T_\alpha(c/h^s)\colon V\rightarrow V$ modeling the left translation action of $x_\alpha(c/h^s)$ on the group $\St(\Phi, A_h)$.

In the construction of this map we use the condition $B/h\cong A/h$ of Theorem~\ref{glueing}, which is equivalent to the condition that for every $k\geq0$ one has $A = Ah^k + B$ and $Ah^k\cap B = Bh^k$. Therefore, we can first decompose $c\in A$ as $c=ah^k+b$ for $k$ large enough, $a\in A$, $b\in B$ and then set
\begin{equation}\label{eq:step2}
T_\alpha(ah^{k-s}+b/h^s)\cdot[u,\,v]=[x_\alpha(b/h^s)u,\,c_{u^{-1}}\big(x_\alpha(ah^{k-s})\big)\cdot v].
\end{equation}
Here we invoke the homomorphism $c_{u^{-1}}$ constructed on Step 1. Obviously, we have to check that the the right-hand side of~\eqref{eq:step2} does not depend on the choices of $k$, $c_g$, etc. This is achieved in Lemma~\ref{well-def}.

{\bf Step~3.} Finally, we prove that the maps $T_\alpha(c/h^s)\colon V\rightarrow V$ constructed on Step~2 satisfy Steinberg relations \eqref{R1}--\eqref{R3}.
As a result, we get a well-defined action of $\St(\Phi,\,A_h)$ on $V$. We also check that this action satisfies the additional requirements~(\ref{additional}) formulated in Proposition~\ref{action}. This is accomplished in Section~\ref{sec:proof-glueing}.

\subsection{Overview of pro-completion}
\label{pro-completion}
We refer the reader to \cite[Section~6.1]{SK06} for an introduction to the formalism of ind-objects and pro-objects.

We start by recalling the definition of the {\it pro-completion} $\catname{Pro}(\mathcal{C})$ of a category $\mathcal{C}$ (cf.~\cite[\S~2.1]{LSV20}). By definition, the objects of $\catname{Pro}(\mathcal{C})$ are all functors $X^{(\infty)}\colon\mathcal{I}^{\mathrm{op}} \to \mathcal{C}$, where $\mathcal{I}$ is an arbitrary small nonempty filtered category. For a fixed pro-object $X^{(\infty)}$ the category $\mathcal{I}$ is called the {\it index category of $X^{(\infty)}$}. For $X^{(\infty)} \in
\catname{Pro}(\mathcal{C})$ and $i \in \mathcal{I}$ we denote by $X^{(i)}$ the value of $X^{(\infty)}$ on $i$. 
We call the images of the arrows of $\mathcal{I}$ under $X^{(\infty)}$ the {\it structure morphisms} of $X^{(\infty)}$. Obviously, $X^{(\infty)}$ depends on $\mathcal I$, however for shortness we suppress it from the notation.

Let $X^{(\infty)}\colon\mathcal{I}\to\mathcal{C}$ and $Y^{(\infty)}\colon\mathcal{J}\to\mathcal{C}$ be a pair of pro-objects.
By definition, a {\it pre-morphism} $$\eta\colon X^{(\infty)} \to Y^{(\infty)}$$ is a pair $\left(\eta^*, \{\eta_j\}_{j\in\mathrm{Ob}(\mathcal{J})}\right)$ consisting of a map $\eta^*\colon \mathrm{Ob}(\mathcal{J})\to\mathrm{Ob}(\mathcal{I})$ and a collection of $\mathcal{C}$-morphisms $\eta_j\colon X^{(\eta^*(j))}\to Y^{(j)}$ such that for every arrow $\psi\colon j' \to j$ in $\mathcal J$ there exists $i$ and arrows $\phi'\colon i\rightarrow\eta^*(j')$ and $\phi\colon i\rightarrow\eta^*(j)$ in $\mathcal I$ such that the composite morphisms $$X^{(i)} \xrightarrow{X^{(\infty)}(\phi')} X^{(\eta^*(j'))} \xrightarrow{\eta_{j'}} Y^{(j')} \quad\text{ and }\quad X^{(i)} \xrightarrow{X^{(\infty)}(\phi)} X^{(\eta^*(j))} \xrightarrow{\eta_j} Y^{(j)} \xrightarrow{Y^{(\infty)}(\psi)} Y^{(j')}$$ are equal. We call the individual $\mathcal{C}$-morphism $\eta_j$ {\it the degree $j$ component} of $\eta$.

Pre-morphisms $\eta, \epsilon \colon X^{(\infty)} \to Y^{(\infty)}$ are declared equivalent if for every $j \in \mathrm{Ob}(\mathcal{J})$ there exists $i \in \mathrm{Ob}(\mathcal{I})$ and arrows $i\rightarrow\eta^*(j)$ and $i\rightarrow\epsilon^*(j)$ in $\mathcal I$ such that the composite morphisms $X^{(i)} \to X^{(\eta^*(j))} \xrightarrow{\eta_j} Y^{(j)}$ and $X^{(i)} \to X^{(\epsilon^*(j))} \xrightarrow{\epsilon_j} Y^{(j)}$ are equal. 

By definition, the set of morphisms $\catname{Pro}\left(\mathcal{C})(X^{(\infty)}, Y^{(\infty)}\right)$ consists of equivalence classes $[\eta]$ of pre-morphisms $\eta \colon X^{(\infty)} \to Y^{(\infty)}$ with respect to the equivalence relation defined above. The composition of a pair of morphisms \([\eta] \colon X^{(\infty)} \to Y^{(\infty)}\) and \([\zeta] \colon Y^{(\infty)} \to Z^{(\infty)}\) is defined as the class $[\zeta \circ \eta]$, where the pre-morphism \(\zeta \circ \eta\) is given by \((\zeta \circ \eta)^*(i) = \eta^*(\zeta^*(i))\) and \((\zeta\circ \eta)_{i} = \zeta_{i} \circ \eta_{\zeta^*(i)}\). There is a fully faithful embedding $\mathcal{C}\to\catname{Pro}(\mathcal{C})$ sending $X \in \mathrm{Ob}(\mathcal{C})$ to the obvious functor $X\colon \{ * \} \to \mathcal{C}$. 

\begin{rem}
Suppose that the category $\mathcal{I}$ is a {\it diagram}, i.e. for any two objects $i$, $j\in\mathcal I$ the hom-set $\mathrm{Hom}_{\mathcal I}(i,\,j)$ is either empty or one-element. Let $\eta \colon X^{(\infty)} \to Y^{(\infty)}$ be a pre-morphism and $\xi\colon \mathrm{Ob}(\mathcal{J})\to\mathrm{Ob}(\mathcal{I})$ be an arbitrary function such that $\mathrm{Hom}_\mathcal{I}\big(\xi(j),\,\eta^*(j)\big)=*$ for all $j\in\mathrm{Ob}(\mathcal J)$. The {\it restriction of $\eta$ with respect to $\xi$} is defined as the pre-morphism $\epsilon\colon X^{(\infty)} \to Y^{(\infty)}$ such that $\epsilon^*=\xi$ and the degree $j$ component $\epsilon_j$ is the composite morphism $X^{(\xi(j))}\to X^{(\eta^*(j))}\xrightarrow{\eta_{j}}Y^{(j)}$. The pre-morphism $\epsilon$ is clearly equivalent to $\eta$. This construction will often be used in the sequel without an explicit reference.
  
The index category of every pro-object that we encounter in the sequel is a diagram (in fact, it is the diagram defined in~Example~\ref{main-example} below).
\end{rem}

\begin{example}
\label{main-example}
a) Let $R$ be an integral domain and $S$ a multiplicative subset of $R$ containing $1$. 
We denote by $\lambda_S$ the canonical localization homomorphism $R \to R_S$.

Notice that $S$ can be interpreted as a filtered category with $\mathrm{Ob}(S) = S$ and $S(s, t) = \{ u \in S \mid su = t \}$.
For shortness, we call a commutative ring without unit an {\it rng} and denote the category of rngs by $\catname{Rng}$.
For $s\in S$ the principal ideal $sR$ can be considered as an rng. We define the pro-rng $R^{(\infty)}$ as the functor $S^\mathrm{op} \to \catname{Rng}$ whose value on an arrow $(s \to t)\in S$ is defined as the obvious embedding of rngs $tR \hookrightarrow sR$. 

b) We reuse the notation $\lambda_S$ for the homomorphism $\St(\Phi, R) \to \St(\Phi, R_S)$ of Steinberg groups induced by the canonical localization homomorphism $R \to R_S$. The {\it Steinberg pro-group} $\St^{(\infty)}(\Phi, R)$ is defined as the functor $S^\mathrm{op} \to \catname{Grp}$ obtained from the pro-rng $R^{(\infty)}$ by postcomposition with $\St(\Phi, -)$ (cf.~\cite[\S~2.4]{LSV20}). We denote by $\pi_S$ the pro-group morphism $\St^{(\infty)}(\Phi, R) \to \St(\Phi, R_S)$ defined by the pre-morphism $\pi^* \colon \{* \} \to S$, $\pi^*(*) = 1$, $\pi_{*} = \lambda_S \colon \St(\Phi, R) \to \St(\Phi, R_S)$. 

c) We denote by $\mathbb N$ the set of nonnegative integers. For a non-nilpotent element $h\in R$ denote by $h^{\mathbb{N}}$ the multiplicative subset $\{ h^n \mid n \in \mathbb{N} \} \subseteq R$. The pro-rng $R^{(\infty)}$ associated to $S=h^\mathbb{N}$ can be depicted as the following diagram of rngs:
\[ \ldots \hookrightarrow h^{n+1} R \hookrightarrow h^n R \hookrightarrow \ldots \hookrightarrow hR \hookrightarrow R, \]
in other words, by definition, $R^{(k)}$ is the principal ideal $h^kR$ and all the structure morphisms of $R^{(\infty)}$ are obvious inclusions.
\end{example}

\subsection{Overview of Steinberg pro-groups}
\label{pro-groups}
We refer the reader to~\cite[\S~2]{LSV20} for a more detailed exposition of the material presented in this subsection. We call the pro-completion of the category $\mathcal C=\catname{Grp}$ from the above subsection {\it the category of pro-groups} $\catname{Pro}(\catname{Grp})$.

For $\alpha \in \Phi$ we define the pro-group morphism $x_\alpha \colon R^{(\infty)} \to \St^{(\infty)}(\Phi, R)$ by means of the following data: $(x_\alpha)^* = \mathrm{id}_S$, $(x_\alpha)_s \colon sR \to \St(\Phi, sR)$, $sr \mapsto x_\alpha(sr)$. These morphisms, which we call {\it root subgroup morphisms}, ``generate'' the Steinberg pro-group $\St^{(\infty)}(\Phi, R)$ from the Example~\ref{main-example}\,b) in the following sense.
\begin{proposition}
\label{generators}
A pair of pro-group morphisms $f, g \colon \St^{(\infty)}(\Phi, R) \to G^{(\infty)}$ coincide if and only if $f \circ x_\beta = g \circ x_\beta$ for all $\beta \in \Phi$.
Moreover, if $\alpha \in \Phi$ is any fixed root and the above equalities hold for all $\beta \in \Phi\setminus\{\alpha\}$, then $f=g$.
\end{proposition}
The first assertion of the above proposition is proven in~\cite[Lemma~2.16]{LSV20} while the second assertion follows from~\cite[Lemma~3.2]{LSV20}.

Now for $u=r/s\in R_S$ $r\in R$, $s\in S$, $\beta\in\Phi$ we define the pro-group morphism $x_\beta(u\cdot-)\colon R^{(\infty)}\rightarrow\St^{(\infty)}(\Phi, R)$ as the equivalence class of pre-morphisms $\eta$ with $\eta^*(t)=(st)$ and $\eta_t(a)=x_\alpha(ua)$ (see~\cite[Subsection~4.2]{LSV20} for more details).
\begin{rem}
For a collection of pro-group morphisms $f_\alpha \colon R^{(\infty)} \to G^{(\infty)}$ satisfying the ``pro-analogues'' of Chevalley commutator identities (see~\cite[Remark~2.15]{LSV20}) there exists a unique morphism $f \colon \St^{(\infty)}(\Phi, R) \to G^{(\infty)}$ such that $f_\alpha = f \circ x_\alpha$ for all $\alpha\in\Phi$. This fact is essential for the construction of homomorphism $\mathrm{conj}\colon \St(\Phi, R_S) \to \mathrm{Aut}(\St^{(\infty)}(\Phi, R))$ from Proposition~\ref{prop:conj-action}, which is outlined in~\cite[Lemma~4.2, Proposition~4.3]{LSV20}.
\end{rem}

 The following proposition is a restatement of~\cite[Proposition~4.3]{LSV20}.
\begin{proposition}\label{prop:conj-action}
 For a reduced irreducible simply-laced root system $\Phi$ of rank $\geq 3$ and a multiplicative system $S\subseteq R$ there exists a group homomorphism $\mathrm{conj}\colon \St(\Phi, R_S) \to \mathrm{Aut}(\St^{(\infty)}(\Phi, R))$ compatible with the obvious conjugation action of $\St(\Phi, R_S)$ on itself, i.e. such that for every $g \in \St(\Phi, R_S)$ the following diagram of pro-groups commutes:
 \begin{equation} \label{eq:conj-pis} \begin{tikzcd} \St^{(\infty)}(\Phi, R) \ar{r}{\mathrm{conj}(g)} \ar{d}{\pi_S} & \St^{(\infty)}(\Phi, R) \ar{d}{\pi_S} \\ \St(\Phi, R_S) \ar{r}{{}^g\!(-)} & \St(\Phi, R_S). \end{tikzcd} \end{equation}
 Moreover, for $g=x_\alpha(u)$, $u\in R_S$, $\alpha\in\Phi$ and $\beta\in\Phi$ one has
 \begin{align}
 \label{conj1}
&\mathrm{conj}(x_\alpha(u)) \circ x_\beta =x_\beta,\quad\text{ for }\alpha+\beta\not\in\Phi\cup0,\\
 \label{conj2}
&\mathrm{conj}(x_\alpha(u)) \circ x_\beta =x_{\alpha+\beta}(N_{\alpha\beta}u\cdot-)\cdot x_\beta(-),\quad\text{ for }\alpha+\beta\in\Phi.
 \end{align}
\end{proposition}
\begin{proof}
The existence of the homomorphism $\mathrm{conj}$ satisfying~(\ref{conj1})--(\ref{conj2}) has been demonstrated in~\cite[Lemma~4.2, Proposition~4.3]{LSV20}.
Thus, we only need to verify the commutativity of~\eqref{eq:conj-pis}. 
Since $\mathrm{conj}$ is a homomorphism, it is enough to consider the case $g = x_\alpha(u)$, $\alpha \in \Phi$, $u \in R_S$. In view of Proposition~\ref{generators}, to verify the commutativity of~\eqref{eq:conj-pis} it is enough to verify the equalities $\pi_S \circ \mathrm{conj}(g) \circ x_\beta= {}^g\!(-) \circ \pi_S \circ x_\beta$ for $\beta\neq -\alpha$.
 It remains to see that the fulfillment of these equalities is a direct consequence of~(\ref{conj1})--(\ref{conj2}). 
 \end{proof}
\begin{rem} \label{rem:conj-action}
Unwinding the definitions, we see that the commutativity of~\eqref{eq:conj-pis} is equivalent to the following condition: for every $g\in \St(\Phi, R_S)$ there exists an element $s_g\in S$ and a homomorphism $c_g \colon \St(\Phi, s_g R) \to \St(\Phi, R)$ such that $\lambda_S(c_g(x)) = g \cdot \lambda_S(i(x))\cdot g^{-1}$ for all $x\in \St(\Phi, s_g R)$. Here $i$ denotes the homomorphism $\St(\Phi, s_g R) \to \St(\Phi, R)$ induced by the obvious embedding $s_gR \hookrightarrow R$. The element $s_g$ and the homomorphism $c_g$ can be determined via the following procedure: choose a representative $\eta$ for the morphism $\mathrm{conj}(g)$ and set $s_g := \eta^*(1)$ and $c_g := \eta_1$ (the degree $1$ component of $\eta$).
\end{rem}

We also need to study basic functoriality properties of $\mathrm{conj}$.
Let $R, R'$ be a pair of domains, $S \subseteq R$, $S' \subseteq R'$ be a pair of multiplicative subsets and let $f \colon R \to R'$ be a homomorphism such that $f(S)\subseteq S'$.
It is clear that $f$ induces a ring homomorphism $\overline{f}\colon R_S \to R'_{S'}$, we also reuse the notation $\overline{f}$ for the induced homomorphism of Steinberg groups.
Now suppose that there is a map $f^* \colon S' \to S$ such that $t$ divides $f(f^*(t))$ for all $t \in S'$ (without loss of generality, we may assume that $f^*(1)=1$ ). The latter condition means, in particular, that $R'_{f(S)} \cong R'_{S'}$. Then $f^*$ and the homomorphisms $\{f^*(t)R \to tR'\}_{t\in S'}$ obtained from $f$ by restricting its domain and codomain specify a pre-morphism $R^{(\infty)} \to R'^{(\infty)}$. We denote by $f^{(\infty)}$ the corresponding morphism of pro-rngs and also use the same notation for the induced morphism of Steinberg pro-groups.

\begin{proposition} \label{prop:functoriality}
In the notation above, the action $\mathrm{conj}$ constructed in Proposition~\ref{prop:conj-action} has the additional property that the following diagram commutes for every $g \in \St(\Phi, R_S)$:
 \begin{equation} \label{eq:conj-finfty} \begin{tikzcd} \St^{(\infty)}(\Phi, R) \ar{rr}{\mathrm{conj}(g)} \ar{d}{f^{(\infty)}} && \St^{(\infty)}(\Phi, R) \ar{d}{f^{(\infty)}} \\ \St^{(\infty)}(\Phi, R') \ar{rr}{\mathrm{conj}\big(\overline{f}(g)\big)} && \St^{(\infty)}(\Phi, R'). \end{tikzcd} \end{equation}
\end{proposition}
\begin{proof}
As before, we may suppose that \(g = x_\alpha(a/s)\) for some \(\alpha \in \Phi\), \(a \in R\), \(s \in S\). By Propostion~\ref{generators}\,b) 
it suffices to check that the long paths in the diagram coincide after precomposition with \(x_\beta \colon R^{(\infty)} \to \St^{(\infty)}(\Phi, R)\) for all \(\beta \neq -\alpha\). This follows from~(\ref{conj1})--(\ref{conj2}) and the commutativity of
\[\begin{tikzcd}
R^{(\infty)} \ar{r}{x_\beta} \ar{d}{f^{(\infty)}} &
\St^{(\infty)}(\Phi, R) \ar{d}{f^{(\infty)}} \\
{R'}^{(\infty)} \ar{r}{x_\beta} & \St^{(\infty)}(\Phi, R'). \end{tikzcd}\]
\end{proof}

\begin{rem} \label{rem:functoriality}
 Unwinding the definitions, we obtain from the commutativity of~\eqref{eq:conj-finfty} that for every $g \in \St(\Phi, R_S)$ there exists $t \in S$ divisible by $s_g$ and $s_{\,\overline f(g)}$ such that the diagram 
 \[ \begin{tikzcd} \St(\Phi, tR) \ar{r} \ar{d} & \St(\Phi, s_g R) \ar{r}{c_{g}} & \St(\Phi, R) \ar{dd}{f} \\
                   \St(\Phi, f^*(s_{\,\overline{f}(g)})R) \ar{d}{f}   & & \\
                   \St(\Phi, s_{\,\overline{f}(g)} R')   \ar{rr}{c_{\,\overline{f}(g)}} & & \St(\Phi, R'). \end{tikzcd} \]
commutes (here $s_g$, $s_{\,\overline{f}(g)}$, $c_g$, $c_{\,\overline{f}(g)}$ have the same meaning as in Remark~\ref{rem:conj-action}).
\end{rem}

Now let us specialize the above results to a particular situation that we will encounter in~\cref{sec:proof-glueing}.

Let $\iota \colon B \to A$ be an injective homomorphism of integral domains and $h \neq 0$ be an element of $B$. 
Denote by $h^{\mathbb{N}}$ the multiplicative subset $\{ h^n \mid n \in \mathbb{N} \} \subseteq B$ as in Example~\ref{main-example}\,c). 
Consider the multiplicative subset $h^\mathbb{N} \subseteq A$ and the corresponding homomorphism of localizations $\overline{\iota} \colon B_h \to A_{h}$. It is clear that if we put $\iota^*(h^n) = h^n$, the subsets $h^\mathbb{N} \subseteq B$, $h^\mathbb{N} \subseteq A$, and the homomorphism $\iota$ satisfy the requirements mentioned before Proposition~\ref{prop:functoriality}. Thus, there exists a morphism $\iota^{(\infty)}$ of pro-rngs $B^{(\infty)} \to A^{(\infty)}$ and the corresponding morphism of Steinberg pro-groups $\iota^{(\infty)}\colon\St^{(\infty)}(\Phi, B) \to \St^{(\infty)}(\Phi, A)$ satisfying the conclusion of Proposition~\ref{prop:functoriality}.

\begin{corollary}
\label{vorcor}
For $\iota \colon B \to A$ and $h \in B$ as above and every $g \in \St(\Phi, B_h)$ there exists $n = n(g) \in \mathbb{N}$ 
and group homomorphisms $c_g \colon \St(\Phi, h^nB) \to \St(\Phi, B)$, $c_{\overline{\iota}(g)} \colon \St(\Phi, h^nA) \to \St(\Phi, A)$
such that the following equalities are fulfilled:
\begin{align}
 \label{eq:coherence} \iota \circ c_g &= c_{\overline{\iota}(g)} \circ \iota, &\\
 \label{eq:strictB} \lambda_h^0 (c_g (x)) &= g \cdot \lambda_h^n(x) \cdot g^{-1}&\text{ for }x \in \St(\Phi, h^nB),\\
 \label{eq:strictA} \lambda_h^0 (c_{\overline{\iota}(g)}(y)) &= \overline{\iota}(g) \cdot \lambda_h^n(y) \cdot \overline{\iota}(g)^{-1}&\text{ for }y \in \St(\Phi, h^nA).
\end{align}
Here $\lambda_h^k$ denotes the homomorphism of Steinberg grops induced by the following composite rng homomorphism $h^kB \hookrightarrow B \to B_h$ (or $h^kA \hookrightarrow A \to A_{h}$), where $k \in \mathbb{N}$.
\end{corollary}
\begin{proof}
 The first equality is a special case of the commutativity of the diagram from Remark~\ref{rem:functoriality}.
 The second and the third equalities follow from the equality discussed in Remark~\ref{rem:conj-action}.
\end{proof}
\begin{rem}\label{rem:indendepence}
 Notice that a priori the homomorphism $c_g$ depends on the choice of a representative for $\mathrm{conj}(g)$. On the other hand, from the definition of equivalence of pre-morphisms it follows that if $c'_g$ is the degree $1$ component for a different representative of $\mathrm{conj}(g)$ (see Remark~\ref{rem:conj-action}) then for sufficiently large $m$ one has $c'_{\overline{\iota}(g)}(x_{\alpha}(ah^m)) = c_{\overline{\iota}(g)}(x_{\alpha}(ah^m))$.
\end{rem}

\subsection{Proof of Theorem~\ref{glueing}} \label{sec:proof-glueing}

Let us briefly recall the definition of the set $V$. Recall that we let $\St(\Phi, B)$ act on $\St(\Phi, B_h) \times \St(\Phi, A)$ on the left via the formula 
\[w \star (u, v) = (u \lambda_h(w)^{-1}, \iota(w) v),\text{ where }w \in \St(\Phi, B),\ u \in \St(\Phi, B_h),\ v \in \St(\Phi, A).\]
We denote by $V$ the orbit set for this action and use the notation $[u, v]$ for the orbit corresponding to a pair $(u, v) \in \St(\Phi, B_h) \times \St(\Phi, A)$.

Our goal is to show that the formula $[u, v] \mapsto \overline{\iota}(u) \cdot \lambda_{h}(v)$ defines a bijection between $V$ and $\St(\Phi, A_{h})$. This clearly implies the exactness of~(\ref{stavrova-sequence}) which is equivalent to the claim of Theorem~\ref{glueing}. 

So far we have realized Step~1 of the plan from Subsection~\ref{plan}, i.e. we have constructed the homomorphisms $c_g$. Now fix some $u \in \St(\Phi, B_h)$, $v \in \St(\Phi, A)$, $\alpha \in \Phi$, and $c/{h^s} \in A_{h}$. 
Applying \cref{vorcor} to the element $u$ we obtain a natural number $n=n(u)$ and homomorphisms \begin{equation} \label{eq:c-homs} c_{u^{-1}}\colon \St(\Phi, h^nB) \to \St(\Phi, B),\ c_{\overline{\iota}(u^{-1})}\colon \St(\Phi, h^nA) \to \St(\Phi, A)\end{equation}
satisfying the identities listed in~\cref{vorcor}.

Now we are ready to move on to Step~2 of the plan, namely, to construct operators $T_\alpha(c/h^s)\colon V\rightarrow V$. Observe that the assumption $A/hA = B/hB$ is equivalent to the condition that for every $k \geq 0$ one has $A = Ah^k + B$ and $Ah^k \cap B = Bh^k$. Now, in the notation above, choose any $k \geq n + s$ and decompose $c \in A$ as
 \begin{equation} \label{eq:decomp} c = ah^k + b\text{ for some }a \in A,\ b \in B. \end{equation}
We define the operator $T_\alpha(c/h^s) \colon V \to V$ via the formula:
\begin{equation}\label{eq:generator-action}
\textstyle
T_\alpha(c/{h^s}) \cdot [u, v] = \bigl[x_\alpha(b/{h^s})\cdot u,\ c_{\overline{\iota}(u^{-1})}(x_\alpha(ah^{k - s})) \cdot v\bigr].
\end{equation}
\begin{lemma}\label{well-def}
The right hand side of~\eqref{eq:generator-action} is independent of either the choice of the decomposition~\eqref{eq:decomp}, the choice of representatives for the morphisms $\mathrm{conj}(u)$, $\mathrm{conj}(\overline{\iota}(u))$ or the choice of a representative for the orbit $[u, v]$.
\end{lemma}
\begin{proof}
First, let us show the independence of the choice of a decomposition for $c$. 
Choose some $l \geq n + s$ and $c = a' h^{l} + b'$. Without loss of generality, we may assume $l \leq k$ so that $a' h^{l} - a h^k = b - b' \in Ah^{l} \cap B = Bh^{l}$ and there exists $d \in B$ such that $b = b' + dh^{l}$. Thus, from~\eqref{eq:strictB} we obtain that
\begin{equation} \label{eq:wd1} \textstyle x_\alpha(b/{h^s})\, u = x_\alpha({b'}/{h^s})\, u \cdot u^{-1}\, x_\alpha(dh^{l-s})\, u =  x_\alpha({b'}/{h^s})\, u \cdot \lambda_h \bigl(c_{u^{-1}}(x_\alpha(dh^{l-s}))\bigr). \end{equation}
Since $A$ is a domain, $h$ is not a zero divisor, consequently from $a'h^l = ah^k + dh^l$ we obtain that $a'h^{l-s} = dh^{l-s} + ah^{k-s}$, thus
\begin{align*}
[x_\alpha(b'/{h^s})\, u,\ c_{\overline{\iota}(u^{-1})}(x_\alpha(a'h^{l - s}))\, v\bigr] = \bigl[x_\alpha(b'/{h^s})\, u,\ \iota\bigl(c_{u^{-1}}(x_\alpha(dh^{l-s}))\bigr)\, c_{\overline{\iota}(u^{-1})}(x_\alpha(ah^{k - s}))\, v\bigr] & \text{ by~\eqref{eq:coherence}} \\ = \bigl[x_\alpha(b/{h^s})\, u,\ c_{\overline{\iota}(u^{-1})}(x_\alpha(ah^{k - s}))\, v\bigr] &\text{ by~\eqref{eq:wd1},}
\end{align*}
 which coincides with the right-hand side of~\eqref{eq:generator-action}.

The independence of the right-hand side of~\eqref{eq:generator-action} of the choice of representatives for $\mathrm{conj}(g)$ and $\mathrm{conj}(\overline{\iota}(g))$ follows from Remark~\ref{rem:indendepence} and the fact that the number $k$ in the decomposition for $c$ could have been chosen arbitrarily large.

Now suppose that $(u', v')$ is another representative for the orbit $[u, v]$, i.e. $u' = u \cdot \lambda_h(w)$, $v' = \iota(w^{-1})\cdot v$ for some $w \in \St(\Phi,\,B)$. It is clear that we can choose a representative $\eta$ for $\mathrm{conj}(\overline{\iota}(\lambda_h(w^{-1})))$ in such a way that $\eta^*(1)=1$ and its first-degree component $\eta_1 \colon \St(\Phi, A) \to \St(\Phi, A)$ coincides with $h \mapsto \iota(w^{-1}) \cdot h \cdot \iota(w)$, moreover, from the definition of the composition of pre-morphisms we conclude that $c_{\overline{\iota}({u'}^{-1})} =  \eta_1 \circ c_{\overline{\iota}(u^{-1})} $. The independence of~\eqref{eq:generator-action} of the choice of $(u, v)$ now follows from the following direct calculation:
\begin{multline*}
 [x_\alpha(b/h^s) \cdot u',\ c_{\overline{\iota}({u'}^{-1})}(x_\alpha(ah^{k-s}))\cdot v'] = [x_\alpha(b/h^s) \cdot u \cdot \lambda_h(w),\ c_{\overline{\iota}({u'}^{-1})}(x_\alpha(ah^{k-s}))\cdot \iota(w^{-1}) v] = \\
 = [x_\alpha(b/h^s) \cdot u \cdot \lambda_h(w),\ \iota(w) \cdot c_{\overline{\iota}(u^{-1})}(x_\alpha(ah^{k-s}))\cdot v] 
 = [x_\alpha(b/h^s) \cdot u,\ c_{\overline{\iota}(u^{-1})}(x_\alpha(ah^{k-s}))\cdot v]. \qedhere
\end{multline*}
\end{proof}

Finally, we move on to Step~3 of the plan, namely we verify that the above operators $T_\alpha(c/h^s)$ specify an action of $\St(\Phi, A_{h})$ on $V$ which satisfies~(\ref{additional}).
This is accomplished in the series of lemmas below:

\begin{lemma}\label{lem:R1} The operators $T_\alpha$ satisfy relations~\eqref{R1}. \end{lemma}
\begin{proof}
We need to show that for $c, c' \in A$ one has $\textstyle
T_\alpha(c'/{h^s}) \cdot T_\alpha(c/h^s) \cdot [u, v] = T_\alpha\bigl((c+c')/h^s\bigr) \cdot [u, v].$
Choose a decomposition $c = ah^k + b$ as in~\eqref{eq:decomp} for some $k \geq n(u) + s$.
We claim that there exists a sufficiently large number $m$ such that $c_{\overline{\iota}(u^{-1}) \cdot x_\alpha(-b/h^s)}(x_\alpha(a'h^{m-s})) = c_{\overline{\iota}(u^{-1})}(x_\alpha(a'h^{m-s}))$ for all $a'\in A$. Indeed, since $\mathrm{conj}$ is a homomorphism, we conclude that
$$
\mathrm{conj}\big(\overline{\iota}(u^{-1}) \cdot x_\alpha(-b/h^s)\big)\circ x_\alpha=\mathrm{conj}(\overline{\iota}(u^{-1}))\circ\mathrm{conj}(x_\alpha(-b/h^s))\circ x_\alpha=\mathrm{conj}(\overline{\iota}(u^{-1}))\circ x_\alpha
$$
by~(\ref{conj1}), in particular, the degree $1$ components of these pro-group maps become equal after the precomposition with a structure morphism $h^{m-s}R\hookrightarrow R$ for $m$ large enough. 
Now choose a decomposition $c' = a' h^m + b'$.
The assertion now follows from the following calculation:
\begin{multline*}
 T_\alpha(c'/{h^s}) \cdot T_\alpha(c/h^s) \cdot [u,\ v] = T_\alpha(c'/h^s) \cdot [x_\alpha(b/h^s)\cdot u,\ c_{\overline{\iota}(u^{-1})}(x_\alpha(ah^{k-s})) \cdot v] = \\
 = [x_\alpha((b + b')/h^s)\cdot u,\ c_{\overline{\iota}(u^{-1}) \cdot x_\alpha(-b/h^s)}(x_\alpha(a'h^{m-s})) \cdot c_{\overline{\iota}(u^{-1})}(x_\alpha(ah^{k-s})) \cdot v] = \\ = [x_\alpha((b + b')/h^s)\cdot u,\  c_{\overline{\iota}(u^{-1})}(x_\alpha(a'h^{m-s} + ah^{k-s})) \cdot v] = T_\alpha\bigl((c+c')/h^s\bigr) \cdot [u, v]. \qedhere
\end{multline*}
\end{proof}

\begin{lemma}\label{lem:R2} The operators $T_\alpha$ satisfy relations~\eqref{R2}. \end{lemma}
\begin{proof}
 Fix a natural number $s$ and a pair of roots $\alpha, \beta \in \Phi$ such that $\alpha+\beta\not\in\Phi\cup0$. Using~(\ref{conj1}) we have
 \begin{align*}
&\mathrm{conj}\big(\,\overline\iota(u^{-1})x_\beta(-b/h^s)\big)\circ x_\alpha =  \mathrm{conj}\big(\,\overline\iota(u^{-1})\big)\circ x_\alpha,
&\mathrm{conj}\big(\,\overline\iota(u^{-1})x_\alpha(-a/h^s)\big)\circ x_\beta =  \mathrm{conj}\big(\,\overline\iota(u^{-1})\big)\circ x_\beta,
 \end{align*}
 and therefore from the definition of composition of pro-group morphisms and the definition of $\mathrm{conj}$ it follows that for any $u \in \St(\Phi, B_h)$ there exists sufficiently large natural number $n$ such that for all $a \in A$, $b\in B$ the equalities
 \begin{align*}
&c_{\,\overline\iota(u^{-1}x_\beta(-b/h^s))}(x_\alpha(ah^{n-s}))= c_{\,\overline\iota(u^{-1})}(x_\alpha(ah^{n-s})),
&c_{\,\overline\iota(u^{-1}x_\alpha(-a/h^s))}(x_\beta(bh^{n-s}))= c_{\,\overline\iota(u^{-1})}(x_\beta(bh^{n-s}))
\end{align*}
 hold. Now decompose $c=ah^n+b$, $c'=a'h^n+b'$, and the direct computation shows that
 \begin{multline*}
 T_\alpha(c/h^s)T_\beta(c'/h^s)[u,\,v]=T_\alpha(c/h^s)[x_\beta(b'/h^s)\cdot u,\,c_{\,\overline\iota(u^{-1})}(x_\beta(a'h^{n-s}))\cdot v]=\\
 =[x_\alpha(b/h^s)\cdot x_\beta(b'/h^s)\cdot u,\,c_{\,\overline\iota(u^{-1}x_\beta(-b'/h^s))}(x_\alpha(ah^{n-s}))\cdot c_{\,\overline\iota(u^{-1})}(x_\beta(a'h^{n-s}))\cdot v]=\\
 =[x_\alpha(b/h^s)\cdot x_\beta(b'/h^s)\cdot u,\,c_{\,\overline\iota(u^{-1})}(x_\alpha(ah^{n-s}))\cdot c_{\,\overline\iota(u^{-1})}(x_\beta(a'h^{n-s}))\cdot v].
 \end{multline*}
 Using the relation~\eqref{R2} in $\mathrm{St}(\Phi,\,B_h)$ and $\mathrm{St}(\Phi,\,A)$, by symmetry we get the claim.
\end{proof}

\begin{lemma} \label{lem:R3-check} The operators $T_\alpha$ satisfy relations~\eqref{R3}. \end{lemma}
\begin{proof}
The idea of the proof is the same as in Lemmas~\ref{lem:R1}--\ref{lem:R2} above: since $T_\alpha(c/h^s)$ are well-defined by Lemma~\ref{well-def}, we can use the decomposition $A=B+h^nA$ with arbitrary large $n$, and this allows to reduce that claim to the commutator formulas in $\mathrm{St}(\Phi,\,B_h)$ and $\mathrm{St}(\Phi,\,A)$.

 Fix a natural number $s$ and a pair of roots $\alpha, \beta \in \Phi$ such that $\alpha+\beta\in\Phi$. Set $\epsilon = N_{\alpha, \beta}$. From the definition of composition of pro-group morphisms and the definition of $\mathrm{conj}$ it follows that for any $u \in \St(\Phi, B_h)$ there exists sufficiently large natural number $n$ such that for all $a \in A$, $b, b'\in B$ the following equalities are simultaneously fulfilled:
 \begin{align}
 c_{ \overline{\iota}((x_\beta(b/h^s) \cdot u)^{-1})}(x_\alpha(ah^{n-s})) &= c_{\overline{\iota}(u^{-1})}\left(x_\alpha(ah^{n-s}) \cdot x_{\alpha+\beta}(\epsilon ab h^{n-2s})\right), \label{eq:R3-check-1}\\
 c_{ \overline{\iota}((x_\alpha(b/h^s) \cdot u)^{-1})}(x_\beta(ah^{n - s})) &= c_{\overline{\iota}(u^{-1})}\bigl(x_{\alpha + \beta}(-\epsilon abh^{n - 2s})\cdot x_\beta(ah^{n - s})\bigr), \label{eq:R3-check-2} \\
 c_{ \overline{\iota}((x_\beta(b'/h^s) \cdot x_\alpha(b/h^s) \cdot u)^{-1})}\bigl(x_{\alpha + \beta}\bigl(a h^{n - 2s}\bigr)\bigr) &= c_{\overline{\iota}(u^{-1})}\bigl(x_{\alpha + \beta}\bigl(a h^{n - 2s}\bigr)\bigr). \label{eq:R3-check-3}
 \end{align} 
Now fix $c, c\in A$. Our goal is to verify the equality 
\begin{equation} \label{eq:R3-to-check} T_\alpha(c/h^s) \cdot T_\beta(c'/h^s) \cdot [u, v] = T_{\alpha+\beta}(\epsilon cc' / h^{2s}) \cdot T_\beta(c'/h^s) \cdot T_\alpha(c/h^s) \cdot [u, v].\end{equation}
As in~\eqref{eq:decomp} we can choose decompositions $c = b + ah^n,\ c' = b' + a'h^n$ for some $a, a' \in A$ and $b, b' \in B$. It is also clear that $cc' = bb' + h^{n}a''$, where $a'' = aa'h^n + ab' + a'b$. Now we can compute the left-hand side of~\eqref{eq:R3-to-check} using~\eqref{eq:generator-action} and~\eqref{eq:R3-check-1} as follows:
\begin{multline*}
 T_\alpha(c/h^s) \cdot T_\beta(c'/h^s) \cdot [u, v] = T_\alpha(c/h^s) \cdot [x_\beta(b'/h^s) \cdot u,\ c_{\overline{\iota}(u^{-1})}(x_\beta(a'h^{n-s})) \cdot v] = \\ = [x_\alpha(b/h^s) \cdot x_\beta(b'/h^s) \cdot u, c_{\overline{\iota}(u^{-1})}(x_\alpha(ah^{n-s}) \cdot x_{\alpha+\beta}(\epsilon ab'h^{n-2s}) \cdot x_\beta(a'h^{n-s}))\cdot v]. \end{multline*}
Similarly, we can compute the right-hand side of~\eqref{eq:R3-to-check} using~\eqref{eq:generator-action} and~\eqref{eq:R3-check-2}--\eqref{eq:R3-check-3}:
\begin{align*}
&T_{\alpha+\beta}(\epsilon cc'/h^{2s}) \cdot T_\beta(c'/h^s) \cdot T_\alpha(c/h^s) \cdot [u, v] = \\ 
&= T_{\alpha+\beta}(\epsilon cc'/h^{2s})\cdot[x_\beta(b'/h^s) \cdot x_\alpha(b/h^s) \cdot u,\,c_{ \overline{\iota}((x_\alpha(b/h^s) \cdot u)^{-1})}(x_\beta(a'h^{n - s}))\cdot c_{\overline{\iota}(u^{-1})}(x_\alpha(ah^{n-s}))\cdot v]=\\
&\ \begin{aligned}
=[&x_{\alpha+\beta}(\epsilon bb'/h^{2s}) \cdot x_\beta(b'/h^s) \cdot x_\alpha(b/h^s) \cdot u,\\
&c_{ \overline{\iota}((x_\beta(b'/h^s) \cdot x_\alpha(b/h^s) \cdot u)^{-1})}\bigl(x_{\alpha + \beta}\bigl(\epsilon a'' h^{n - 2s}\bigr)\bigr)\cdot c_{\overline{\iota}(u^{-1})}\bigl(x_{\alpha + \beta}(-\epsilon a'bh^{n - 2s})\cdot x_\beta(a'h^{n - s})\cdot x_\alpha(ah^{n-s}\bigr)\cdot v]=
\end{aligned}\\
&= [x_{\alpha+\beta}(\epsilon bb'/h^{2s}) \cdot x_\beta(b'/h^s) \cdot x_\alpha(b/h^s) \cdot u,\, c_{\overline{\iota}(u^{-1})}(x_{\alpha+\beta}(\epsilon h^{n-2s}(a''- a'b)) \cdot x_\beta(a'h^{n-s}) \cdot x_\alpha(ah^{n-s}))\cdot v]. \end{align*}
Finally, it follows from~\eqref{R3} in $\mathrm{St}(\Phi,\,B_h)$ and $\mathrm{St}(\Phi,\,A)$ that the right-hand sides of the last two formulae coincide. Clearly, 
$$
x_\alpha(b/h^s) \cdot x_\beta(b'/h^s)=x_{\alpha+\beta}(\epsilon bb'/h^{2s}) \cdot x_\beta(b'/h^s) \cdot x_\alpha(b/h^s),
$$
and therefore the first components of the expressions coincide, and using
$$
x_\alpha(ah^{n-s}) \cdot x_\beta(a'h^{n-s})=x_{\alpha+\beta}(\epsilon aa'h^{2n-2s}) \cdot x_\beta(a'h^{n-s}) \cdot x_\alpha(ah^{n-s}),
$$
it remains to observe that $ab'h^{n-2s}+aa'h^{2n-2s}=(a''-a'b)h^{n-2s}$ to conclude that the second components also coincide.
\end{proof}

It follows that the action of $\St(\Phi, A_{h})$ on $V$ is well-defined. It remains to prove the following result.

\begin{lemma} The constructed action satisfies conditions~(\ref{additional}). 
\end{lemma}
\begin{proof}
Indeed, using the fact that the operators $T_\alpha(a)$ are well-defined by Lemma~\ref{well-def}, for $a\in A$ we can choose a decomposition $a=ah^0+0$ and for $u=1\in\mathrm{St}(\Phi,\,B_h)$ we can choose $c_{1}$ to be the identity map. Then by the very definition
$$
T_\alpha(a)\cdot[1,\,v]=[1,\,x_\alpha(a)\cdot v],
$$
and by induction we conclude that $\lambda_h(v)\cdot[1,\,1]=[1,\,v]$ for any $v\in\mathrm{St}(\Phi,\,A)$. Next, for $b\in B$ we can use a decomposition $b=0\cdot h^0+b$, then by definition
$$
T_\alpha(b/h^s)\cdot[u,\,v]=[x_\alpha(b/h^s)\cdot u,\,c_{\,\overline\iota(u^{-1})}(1)\cdot v]=[x_\alpha(b/h^s)\cdot u,\,v],
$$
and by induction we conclude that $\overline\iota(u)\cdot[1,\,v]=[u,\,v]$ for any $u\in\mathrm{St}(\Phi,\,B_h)$.
\end{proof}

Therefore we conclude that the sequence~(\ref{stavrova-sequence}) is exact by Corollary~\ref{corollary-in-the-end}, and therefore Theorem~\ref{glueing} holds.

\subsection*{Funding}
The work on \S~2 was supported by the Ministry of Science and Higher Education of the Russian Federation, agreement No. 075-15-2019-1619. The work on \S~3 was supported by the Russian Science Foundation grant No. 19-71-30002. The work on \S~4 was supported by the Foundation for the Advancement of Theoretical Physics and Mathematics ``BASIS''.
The first-named author is a winner of the Young Russian Mathematics contest and would like to thank its sponsors and the jury.

\subsection*{Acknowledgements}
 The authors would like to thank A. Stavrova for suggesting several ideas used in this paper and the anonymous referees for suggesting a number of comments which helped to improve this article. We would also like to thank P.~Gvozdevsky, S.~Ivanov, A.~Sawant, V.~Sosnilo and N.~Vavilov for their useful comments and interest in this work.

\end{document}